%% file: BFM_Tori_main_file-v3.tex
\documentclass[a4paper,reqno]{amsart}
\usepackage{amsmath,amssymb,comment}
\usepackage{amsfonts}
\usepackage{paralist}
\usepackage{latexsym}
\usepackage{epsfig}
\usepackage{graphicx}
\usepackage{url}
\usepackage{color}
\usepackage{mathabx}
\usepackage[colorlinks,unicode]{hyperref}

\usepackage[utf8]{inputenc}

\newtheorem{theorem}{Theorem}[section]

\newtheorem{claim}[theorem]{Claim}

\newtheorem{corollary}[theorem]{Corollary}

\newtheorem{lemma}[theorem]{Lemma}

\newtheorem{proposition}[theorem]{Proposition}
\newtheorem{remark}[theorem]{Remark}

\def\N{\mathbb{N}}
\def\R{\mathbb{R}}
\def\C{\mathbb{C}}
\def\Z{\mathbb{Z}}
\def\TT{\mathbb{T}}
\def\CS{\mathbb{H}}

\def\th{\theta}

\def\e{\delta}

\def\Dif{{D}}
\begin{sloppypar}
\end{sloppypar}
\def\Spec{{\rm Spec}\,}
\def\Lip{{\rm Lip}\,}
\def\avg{{\rm avg}\,}
\def\vol{{\rm vol}\,}
\def\Re{{\rm Re}\,}
\def\Im{{\rm Im}\,}
\def\L{\mathcal{L}}
\def\O{\mathcal{O}}
\def\U{\mathcal{U}}

\def\E{\mathcal{E}}
\def\K{\mathcal{K}}
\def\Ein{E^{\leq}}
\def\T{\mathcal{T}}
\def\S{\mathcal{S}}
\def\SD{\mathcal{S}\mathcal{D}}
\def\SDVF{\mathcal{S}\mathcal{D}}

\def\Dy{\partial_{y}}

\newcommand{\Id}{\mathrm{Id}}

\newcommand{\HH}{\mathcal{H}}
\newcommand{\GGG}{\mathcal{G}}
\newcommand{\X}{\mathcal{X}}

\def\f{f^\e}
\def\g{g^\e}
\def\h{h^\e}
\def\J{J}

\def\Aee{\text{\rm \ae}}

\title[Parabolic Tori]{
Whiskered parabolic tori in the planar $(n+1)$-body problem}

\author[Inmaculada~Baldom\'a, Ernest~Fontich and Pau~Mart\'{\i}n]{}
\email{immaculada.baldoma@upc.edu}
\email{fontich@ub.edu}
\email{p.martin@upc.edu}

\thanks{
I.B and P.M. have been partially supported by the Spanish
MINECO-FEDER Grant MTM2015-65715-P and the Catalan Grant 2014SGR504.
The work of E.F. has been partially supported by the Spanish
Government grant MTM2016-80117-P
(MINECO/FEDER, UE) and the Catalan Government grant 2017-SGR-1374.
Also all authors have been partially supported by the Maria de Maeztu project MDM-2014-044.
}
\date{\today}
\begin{document}

\maketitle

\centerline{\scshape Inmaculada Baldom\'a}
\medskip
{\footnotesize
 \centerline{Departament de Matem\`{a}tiques,}
\centerline{Universitat Polit\`{e}cnica de Catalunya,}
\centerline{Av. Diagonal 647, 08028 Barcelona, Spain}
}
\medskip

\centerline{\scshape Ernest Fontich}
{\footnotesize
 \centerline{Departament de Matem\`{a}tiques i Inform\`atica,}
 \centerline{Institut de Matem\`atiques de la Universitat de Barcelona (IMUB),}
 \centerline{Barcelona Graduate School of Mathematics (BGSMath),}
 \centerline{Universitat de Barcelona (UB),}
 \centerline{Gran Via 585,
 08007 Barcelona, Spain} }
\medskip

\centerline{\scshape Pau Mart\'{\i}n}
{\footnotesize
 \centerline{Departament de Matem\`{a}tiques,}
\centerline{Universitat Polit\`{e}cnica de Catalunya}
\centerline{Ed.~C3, Jordi Girona 1--3, 08034 Barcelona, Spain}
}

\bigskip

\begin{abstract}
The planar $(n+1)$-body problem models the motion of $n+1$ bodies in the
plane under their mutual Newtonian gravitational attraction forces. When $n\ge 3$, the question about final motions, that is, what are the possible limit motions in the planar $(n+1)$-body problem when $t\to \infty$, ceases to be completely meaningful due to the existence of non-collision singularities.

In this paper we prove the existence of solutions of the planar $(n+1)$-body problem which are defined for all forward time and tend to a parabolic motion, that is, that one of the bodies reaches infinity with zero velocity while the rest perform a bounded motion.

These solutions are related to whiskered parabolic tori at infinity, that is, parabolic tori with stable and unstable invariant manifolds which lie at infinity. These parabolic tori appear in cylinders which can be considered ``normally parabolic''.

The existence of these whiskered parabolic tori is a consequence of a general theorem on parabolic tori developed here. Another application of our theorem is a conjugation result for a class of skew product maps with a parabolic torus with its normal form generalizing results of Takens and Voronin~\cite{Takens,Voronin}.
\end{abstract}

\tableofcontents

\input{BFM_Tori_Introduction-v3}

\input{BFM_Tori_Statementes_and_main_results-v3}

\input{BFM_Tori_Celestial_mechanics-v3}

\input{BFM_Tori_Proofs_Reduit-Diophantine-v3}



\bibliography{references}
\bibliographystyle{alpha}

\end{document}

%% file: BFM_Tori_Introduction-v3.tex
\section{Introduction}

In the study of the $(n+1)$-body problem, in celestial mechanics, one important question is about the possible final motions, i.e., the possible ``limit states'' of a solution of the $(n+1)$-body problem as time goes to $\pm \infty$. In the case of the three body problem, Chazy~\cite{Chazy22} (see also~\cite[Chap. 2]{ArnoldKN88}) gave a complete classification of the possible final motions, with seven options: if all the bodies reach infinity, their motion could be (i) hyperbolic, when all the bodies reach infinity with positive velocity, (ii) hyperbolic-parabolic, when at least one of the bodies reaches infinity with vanishing velocity and another does it with  positive velocity, or (iii) parabolic, when all the bodies reach infinity with zero velocity; (iv) parabolic-elliptic and  (v) hyperbolic-elliptic are the cases when one of the bodies reaches infinity with zero or non-zero velocity, resp., while the others tend to an elliptic motion; (vi) bounded and, finally, (vii) oscillatory, when at least one body goes closer and closer to infinity while always returning to a fixed neighborhood of the other two. Chazy knew examples of all these types of motion, except the oscillatory ones. The existence of the latters, in the case of restricted three body problem (a simplified model of the three body problem where one of the masses is assumed to be zero, not affecting the other two masses, which thus describe Keplerian conics) was first proven for the the Sitnikov problem (a configuration where the bodies with non-zero mass, the primaries, describe ellipses while the third body moves in the line through their center of mass and orthogonal to the plane where the movement of the primaries takes place)  by Sitnikov~\cite{Sitnikov60} and, later, by Moser~\cite{Moser01}. In the restricted planar circular three body problem, oscillatory motions were obtained first by Llibre and Sim\'o in~\cite{SimoL80}. More recently, in the restricted planar circular, it was shown in~\cite{GuardiaMS14} that there are oscillatory motions for all values of the mass parameter.

The existence of oscillatory motions in all these instances of the restricted or full planar three body problem is strongly related to some invariant objects at ``infinity with zero velocity'', either fixed points or periodic orbits, and their stable and unstable invariant manifodls. It is important to remark that these invariant objects, related to parabolic motions, are also ``parabolic'' in the sense that the linearization of the vector field on them vanishes identically and thus all its eigenvalues are~$0$. However, although these points or periodic orbits are not hyperbolic, they do have ``whiskers'' in the traditional sense of hyperbolic invariant objects, that is, stable and unstable invariant manifolds which locally govern the dynamics close to the invariant object and whose intersections are in the heart of the global phenomena from which the oscillatory motions arise. For instance, in the restricted circular planar three body problem, the ``parabolic infinity'' is foliated by fixed points while
in elliptic case, the objects at infinity are periodic orbits. In both cases, the union of these invariant objects is a ``whiskered parabolic cylinder''. In the planar restricted elliptic three body problem, it is proven in~\cite{DelshamsKRS14} the existence of Arnold diffusion along this cylinder.
In~\cite{GuardiaMSS17}, oscillatory orbits related to these parabolic periodic orbits are found for small eccentricity and any value of the mass parameter. See also \cite{McGehee73,Robinson84,Robinson15}.
Moeckel in~\cite{Moeckel2007} uses orbits between near collisions and the parabolic infinity in the three body problem to find symbolic dynamics. In~\cite{Boscaggin2017}, the authors consider the $n$-center problem and prove using variational methods the existence of parabolic trajectories having prescribed asymptotic forward and backward directions.

When one considers the $(n+1)$-body problem with $n\ge 3$,  due to the existence of non-collision singularities, the flow of the system is no longer complete. However, for solutions which are defined for all forward time, the question about their final motion is still of interest. Statements on final motions in the $(n+1)$-body problem, for $n\ge 3$, are scarce.
The most celebrated result in this situation is the existence of bounded motions, by Arnold~\cite{Arnold63b} in the planar case, later generalized to the spatial case by Herman and F\'ejoz~\cite{Fejoz04} and by Chierchia and Pinzari~\cite{ChierchiaP11}. These bounded motions correspond to KAM tori of maximal dimension.

The purpose of this paper is to study the generalization of the invariant parabolic points or periodic orbits at infinity and their stable and unstable manifolds to the case of the planar $(n+1)$-body problem, $n \ge 3$. We consider ``Diophantine parabolic tori'' at ``infinity'', for any $n \ge 3$, and show that these tori do have ``whiskers''(see Theorem~\ref{thm:tori_at_infinity} for the precise statement), which are analytic.
We remark that these tori are not isolated. On the contrary, they appear as one parameter families, thus creating parabolic cylinders foliated by Diophantine tori. The invariant manifolds of the cylinders are the union of the invariant manifolds of the parabolic tori. The importance of these structures is twofold. On the one hand, it provides the following corollary related to final motions in the $(n+1)$-body problem.

\begin{claim}[after Theorem~\ref{thm:tori_at_infinity}]
For any $n\ge 2$, the planar $(n+1)$-body problem has parabolic-bounded motions, that is, solutions such that the relative position of one the bodies to the center of mass of the others goes infinity with zero velocity while the relative positions of rest of the bodies around their center of mass evolve in a bounded motion.
\end{claim}

In Section~\ref{sec:N+1bodyproblem} we clarify the bounded motions the above solutions are related to. Roughly speaking, these bounded motions are linked essentially (but not uniquely) to the maximal KAM tori given by Arnold's theorem and, hence one can only assume their existence in the planetary case, that is, when all except one of the masses are small. F\'ejoz~\cite{Fejoz14} announced in 2014 that there are KAM tori for arbitrary masses if the semi-major axis are chosen appropriately, which would then imply the existence of parabolic-bounded motions in the planar $(n+1)$-body problem for any value of the masses. See Remark~\ref{Diophantine_tor_in_the_n_body_problem}. Other sources of maximal KAM tori are those surrounding normally elliptic periodic orbits. For instance, among the $n$-body choreographies~\cite{ChenM00} (see also~\cite{Moo93}), there is numerical evidence that the figure eight in the three body problem is normally elliptic (see~\cite{Simo02b}).

On the other hand, although it is outside the scope of this paper, the existence and regularity we obtain here of these structures allows to quantitatively describe the passage of an orbit close to infinity, which is a first step to obtain diffusion or oscillatory orbits along them. It should be noted that in the $(n+1)$-body problem it is not possible to find diffusion orbits along the cylinders we obtain in this paper because each torus lies in a different level of the full angular momentum (see Remark~\ref{remark:no_diffusion}). This  is not an obstacle to obtain oscillatory orbits. Diffusion would only be possible jumping among different cylinders. This obstruction is not present in restricted planar $(n+1)$-body problem, where these tori are also present. An interesting question is if in this last case is possible to find Arnold diffusion or oscillatory orbits along the parabolic cylinders (when $n=2$ this was done in~\cite{DelshamsKRS14} and~\cite{GuardiaMSS17}, resp., for small values of the eccentricity).

%

The proof of this result follows from a general statement on parabolic tori, which can be applied to
the restricted planar and full $(n+1)$-body problem, in Section~\ref{sec:N+1bodyproblem}.  More concretely, the statement applies to analytic maps of the form
\[
f:
\begin{pmatrix}
x \\ y \\ \theta
\end{pmatrix}  \mapsto
\begin{pmatrix}
x + \O(\|(x,y)\|^N) \\ y + \O(\|(x,y)\|^N) \\ \theta + \omega + \O(\|(x,y)\|^L)
\end{pmatrix},
\]
or analogous vector fields, where $N, L > 1$ are natural numbers, $(x,y)$ belong to a neighborhood of the origin in $\R\times \R^m$, $\theta \in \TT^d$, the $d$-dimensional torus, and $\omega \in \R^d$ satisfies a \emph{Diophantine condition} (condition~\eqref{DefDiophantinemap}, in the case of maps, \eqref{DefDiophantineflow}, for flows).
We will assume that the map depends analytically on parameters. For this kind of maps, the set $\T = \{ x=0, \; y= 0\}$ is an invariant $d$-dimensional torus, and $f_{\mid \T}: \theta \mapsto \theta + \omega$ is a rigid rotation. We will give conditions on the terms of degree $N$ and $L$ of $f$ under which $\T$ possesses ``whiskers'', that is, $(1+d)$-dimensional stable and unstable manifolds which will parameterize the stable and unstable sets of $\T$ in certain regions with $\T$ at their boundary.
See~\eqref{system} for the case of maps and~\eqref{vectorfield_truemanifold}, for flows, for the whole set of hypotheses. With respect to their regularity,
the stable and unstable manifolds will be analytic in some complex domain, with the invariant torus at its boundary, and $C^{\infty}$ at $\T$.

The proof of the existence of the stable invariant manifold is performed in two steps and is based on the parametrization method.  See~\cite{CabreFL03a,CabreFL03b,CabreFL05,HaroCFLM16} an the references therein for the parametrization method. See also~\cite{BFdLM2007,BFM2015a,BFM2015b,BFM17} for the application of the parametrization method in the case of parabolic fixed points.

The first step is presented as an \emph{a posteriori} result in Theorem~\ref{thetruemanifold}, that is, assuming that one can find a ``close to
invariant''manifold satisfying certain hypotheses, then there is a true invariant manifold nearby. It is worth to remark that this \emph{a posteriori} result does not need the frequency of the rotation on the torus to be Diophantine if some lowest order terms do not depend on~$\theta$, as is the case of many applications. Under these last assumptions, the existence of
a ``close to invariant'' manifold implies the existence of a true manifold even if the frequency vector is resonant.

The second step is devoted to the computation of a ``close to invariant'' manifold, in Theorem~\ref{formalmanifold}.
This approximation of the invariant manifold is a polynomial in a one-dimensional variable with coefficients depending on~$\theta$. Of course, there is quite a lot of freedom in the choice of the coefficients. The Diophantine condition on~$\omega$ is used at this point, where a finite number or small divisor equations appear. It should be noted that
if $\omega$ is resonant but the cohomological equations can be solved up to a given order, then an approximation of the invariant manifold can be found to that order. If this order is large enough, the \emph{a posteriori} Theorem~\ref{thetruemanifold} applies and a true manifold is obtained. However, the degree of regularity of this manifold at the torus will be finite.

The computation of this approximation is simpler if a normal form procedure is applied to the original map. Under the standing hypotheses, the map can be assumed to have a much simpler form. However, we have chosen to deal with the original map for two reasons. The first one concerns the size of the domains of analyticity of the manifolds we obtain. They are essentially those of the map to which one applies the procedure. Normal form procedures shrink this domain. The second one
is to present the algorithm of the computation of the approximate manifold in its full generality, in a way that can be implemented numerically in a given system. The algorithm can be useful in numerical explorations far from perturbative settings and computer assisted proofs.

As a consequence of our claims and techniques, we obtain
the conjugation of a class of skew product maps with a parabolic torus with its normal form, extending some of the results by Takens \cite{Takens} and Voronin \cite{Voronin} to parabolic tori (see Corollary~\ref{conjres}).

The paper is organized as follows. In Section~\ref{sec:mainresults} we state the notation and the main results in this work in both settings, maps and quasiperiodic vector fields. In Section~\ref{sec:N+1bodyproblem} we apply our theory to the restricted and full planar $(n+1)$-body problem.
Next, in Sections~\ref{sec:proofsmaps} and~\ref{sec:proofsedos}, we provide the proofs
of our results for maps and quasiperiodic vector fields, respectively.

%% file: BFM_Tori_Statementes_and_main_results-v3.tex
\section{Statement and main results}
\label{sec:mainresults}
This section is devoted to enunciate properly the results in this work about the existence of invariant manifold of normally parabolic invariant tori in
a very general setting.
For the sake of completeness we deal with two scenarios: analytic maps in Section~\ref{subsec:mapcase} and analytic quasi periodic differential equations,
in Section~\ref{subsec:flowcase}.

The results we are interested in can be splitted into two categories: the first one is the so called \textit{a posteriori} results which, assuming
good enough approximation
of the invariant object (in our case an invariant parabolic manifold) and certain non-degeneracity conditions,
provides a true invariant object close to the approximated one, the second one deals with the obtaining of computable algorithms to find the mentioned approximation.

Besides the existence of the invariant manifold, we are also interested in its regularity with respect to both space variables and parameters.
As it is usual in the parabolic case, at the fixed point, we can not guarantee analyticity generically.
However, we can prove analyticity on open sectors having the fixed point as a vertex.

\subsection{Notation}\label{subsec:notation}
In this short section we present some common notation to both settings maps and flows.

First we introduce the sets we work with and the definition of Diophantine vector:
\begin{itemize}
\item Open ball: we represent by $B_\rho$ the open ball of center $0$ and radius $\rho$. From the context it will be clear in which space is contained.
\item The complex strip: for a given $\sigma>0$, we introduce
$$
\CS_\sigma=\{ z\in \C \mid\, |\Im z |<\sigma\}.
$$
\item The real and complex  $d$-torus: the real torus is $\TT^d = (\R/ \Z)^d$. Given $\sigma>0$ the complex torus is
$$
\TT^d_\sigma=\{ \th=(\th_1, \dots,\th_d) \in (\C/\Z)^d\mid \, |\Im \th_i|<\sigma\}.
$$

\item Given $U\in \R^{k}$, we denote by $U_{\C}$ a complex neighbourhood of $U$.
\item  The open complex sector: given $\beta>0$ and $\rho>0$ we introduce
\begin{equation*}
S= S(\beta,\rho)=\{ t= r e^{{\rm i} \varphi} \in \C\ \mid \ 0<r<\rho, \; |\varphi|< \beta/2\}.
\end{equation*}
Note that $0\notin S(\beta,\rho)$. We will omit the parameters $\beta$, $\rho$ and $\sigma$ in $S$ and $\TT^d$ when they will be clear from the context.
\item
$\omega \in \R^d$ is Diophantine if there exist $c>0$,
\begin{enumerate}
\item  and $\tau\geq d$ such that, in the map context:
\begin{equation}
\label{DefDiophantinemap}
|\omega \cdot  k -l| \ge c |k|^{-\tau}, \qquad \text{for all}\qquad k \in \Z^{d}\backslash\{0\},\, l\in \Z,
\end{equation}
\item  and $\tau\geq d-1$ such that, in the flow context:
\begin{equation}
\label{DefDiophantineflow}
|\omega \cdot  k | \ge c |k|^{-\tau}, \qquad \text{for all}\qquad k \in \Z^{d}\backslash\{0\},
\end{equation}
\end{enumerate}
where $|k| = |k_1| + \cdots +|k_{d}|$ and $\omega \cdot k $ denotes the scalar product.

Notice that $\omega \in \R^d$ is Diophantine in the sense of flows if and only if $\big (\omega_2/\omega_1, \cdots, \omega_d \ \omega_1\big )$
is Diophantine in the sense of maps.
\end{itemize}

Concerning averages we introduce the following definition for maps:
\begin{itemize}
\item given $U\subset \R^{1+m}$ such that $0\in U$, $\Lambda \subset \R^p$ and $h:U \times \TT^{d}\times \Lambda \to \R^k$ we define the average with respect to $\th$:
$$
\overline{h}(z,\lambda) =\avg(h)(z,\lambda)=\frac{1}{\vol(\TT^d)} \int_{\TT^d} h(z, \theta,\lambda) \, d\theta, \qquad (z,\lambda)\in U \times \Lambda
$$
and the oscillatory part
$$
\widetilde h(z, \theta,\lambda) = h(z, \theta,\lambda) - \overline{h}(z,\lambda).
$$
\end{itemize}

With respect to the flow case, given $U\subset \R^{1+m}$ such that $0\in U$, $\Lambda \subset \R^p$ and
$h:U \times \TT^{d}\times \R\times \Lambda \to \R^k$.
\begin{itemize}
\item we say that $h$ is quasiperiodic with respect to $t$ if
there exist a vector of frequencies $\nu=(\nu_1 , \cdots, \nu_{d'})$ and a function
$\widehat h : U \times \TT^{d}\times \mathbb{T}^{d'} \times \Lambda\to \R^k$ such that
$$
h(z,\th,t,\lambda) = \hat h(z,\th,\nu t,\lambda).
$$
We will refer to $\nu$ as the time frequencies of $h$.
\item  We denote the average of $h$ by
$$
\overline h(z,\lambda) = \avg(h)(z,\lambda)=\frac{1}{\text{vol}(\mathbb{T}^{d+d'})}
\int_{\mathbb{T}^{d+d'}}\hat h(z,\theta,\theta',\lambda) \,d\theta \, d\theta'
$$
and the oscillatory part by
$$
\widetilde h(z, \theta,t,\lambda) = h(z, \theta,t,\lambda) - \overline{h}(z,\lambda).
$$
\end{itemize}

Finally we introduce the following general notation and conventions.

\begin{itemize}
\item Let $U\subset \R^k\times \TT^{d}$ and $V\subset \R^{k'}\times \TT^{d'}$. If $\lambda \in \Lambda$ is a parameter, $g:U\times \Lambda \to V$ and
$h:V\times \Lambda \to \R^{k''} \times \TT^{d''}$,
then $f=h\circ g$ is defined by
$$
f(z,\lambda)=h(g(z,\lambda),\lambda).
$$
When dealing with vector fields, sometimes, concerning compositions, $t$ will be considered as a parameter.
\item Let $U\subset \R^{1+m}$, $W\subset \R^{m'}$ and $h:U\times W \to \R^\ell$. For $l\in \N\cup \{0\}$, $k\in (\N\cup \{0\})^m$,
$$
h_{lk}(w) x^ly^k = \frac{1}{l! \, k!} \partial ^l_x\partial ^k_y h(0,0,w)x^ly^k,\qquad (x,y)\in U\subset \R^{1+m}, \quad w\in W,
$$
the corresponding monomial in its expansion around $(x,y)=(0,0)$ using the standard convention $k!=k_1! \dots k_m ! $.
\item Let $U\subset \R^{1+m}$, $W\subset \R^{m'}$ and $h:U\times W \to \R^\ell$. We write $h(z,w) = \O(\|z\|^l)$ if and only if $h(z,w)=\O(\|z\|^l)$
uniformly in $w$.
We also write $h=\O(\|z\|^l)$.
\item If $Z\in \R^{1+m}\times \TT^{d}$ or $Z$ is a function taking values in $\R^{1+m}\times \TT^{d}$, we will write $Z_{x}, Z_{y}, Z_{\theta}$,
the projection over the subspaces generated by the variables $x,y,\theta$ respectively. Also we will use the notation  $Z_{x,y}=(Z_x,Z_y)$ as well as
an analogous notation for any other
combination of the variables $(x,y,\theta)$. Analogously for functions $Z(x,y,\th,\tau)$.
\item We will omit, to avoid cumbersome notation, the dependence of the functions we will work with on some of the variables when there is no danger of confusion.
\item We also make the convention that if $p>q$, the sum $\sum_{l=p}^q$ is void.
\end{itemize}

\subsection{Results for maps}\label{subsec:mapcase}
First we introduce the maps under consideration. Let $\U\subset \R\times \R^{m} $ be an open neighborhood of $0= (0,0)\in \U$ and
$\Lambda \subset \R^p$. We consider $F:  {\U}\times \TT^{d} \times \Lambda\longrightarrow \R\times \R^{m} \times \TT^{d}$, the maps defined by
\begin{equation}\label{system}
F \begin{pmatrix} x \\ y \\ \theta \\ \lambda
\end{pmatrix}
=
\begin{pmatrix} x - a(\theta,\lambda) x^N  +   f_N(x,y,\theta,\lambda) + f_{\geq N+1} (x,y,\theta,\lambda) \\
y  + x^{N-1} B(\theta,\lambda) y+ g_{N}(x,y,\theta,\lambda)+ g_{\geq N+1}(x,y,\theta,\lambda) \\
\theta +\omega  +  h_P(x,y,\theta,\lambda )+ h_{\geq P+1}(x,y,\theta,\lambda)
\end{pmatrix}
\end{equation}
with
\begin{enumerate}
\item [(i)] $N, P$ are integer numbers,

\item [(ii)]  $N\geq 2$, $ P\ge 1$,

\item [(iii)]$\omega \in \R^d$,


\item [(iv)]$f_N(x,y,\theta,\lambda)$ and $g_{N}(x,y,\theta,\lambda)$ are homogeneous polynomials of degree $N$ in the variables $x,\,y$
with coefficients depending on $(\th, \lambda) \in \TT^d \times \Lambda$.
In the same way, $h_P $ is a homogeneous polynomial of degree $P$ in the variables $x,\,y$.
We also assume that $ f_N(x,0,\theta,\lambda)=0$, $g_N(x,0,\theta,\lambda)=0$ and $\Dif_y  g_{N}(x,0,\theta,\lambda)=0$,

\item [(v)]$f_{\geq N+1}$ and $g_{\geq N+1}$ have order $N+1$ (the function and its derivatives with respect to $(x,y)$ vanish
up to order $N$ at $(0,0,\theta,\lambda)$) and
$h_{\geq P+1}$ has order~$P+1$.
\end{enumerate}

It is clear that the set
\begin{equation}\label{torusinv}
\mathcal{T}^d:=\{(0,0,\theta) \in \U\times \TT^d\}
\end{equation}
is an invariant torus of $F$, i.e. for any $\lambda\in \Lambda$, $F(\mathcal{T}^d,\lambda)\subset \mathcal{T}^d$, and all its normal directions are
parabolic. In this work we want to study whether this parabolic torus
has an associated invariant manifold.
To do so we will use the parameterization method, see~\cite{CabreFL03a,CabreFL03b,CabreFL05,BFdLM2007,HaroCFLM16,BFM2015a,BFM2015b}.
This method consists in looking for
$K(x,\theta,\lambda)$, $R(x,\theta,\lambda)$ such that $K(0,\theta,\lambda)=(0,0,\theta)\in \R \times \R^m \times \TT^d$, $R(0,\theta,\lambda)=0$
and satisfying the invariance equation
$$
F (K(x,\theta,\lambda),\lambda) = K(R(x,\theta,\lambda),\lambda).
$$
We will restrict ourselves to obtain one dimensional attracting manifolds so that we will consider
$K_x(x,\theta,\lambda) = x+\O(|x|^2)$ where $x$ is a one dimensional variable.

The first claim is an \textit{a posteriori} result.
\begin{theorem}[\textit{A posteriori} result] \label{thetruemanifold}
Let $F$ be a real analytic map having the form~\eqref{system} satisfying conditions (i)-(v). Assume that
	\begin{enumerate}
		\item $P\geq N$,
		\item either $\omega$ is Diophantine or the functions $a,B$ do not depend on $\theta$.
		\item $\overline a(\lambda)>0$ for $\lambda \in \Lambda$,
		\item $\Re \Spec \overline B(\lambda)>0$ for $\lambda \in \Lambda$.
	\end{enumerate}
	Let $Q\geq N$ and assume that, for some $\beta_0,\rho_0,\sigma_0>0$ and $\Lambda_{\C} \subset \C^p$, there exist
	$K^{\leq}: S(\beta_0,\rho_0)\times \TT^d_{\sigma_0} \times \Lambda_{\C}\to \C^{1+m}\times \TT^d_{\sigma_0}$ and
	$R^{\leq}: S(\beta_0,\rho_0)\times \TT^d_{\sigma_0} \times \Lambda_{\C} \to \C\times \TT^d_{\sigma_0}$,
satisfying that
	\begin{equation*}
	\| K^{\leq}_x (x,\th,\lambda) -x \| \leq C |x|^2, \qquad
	\| K^{\leq}_y (x,\th,\lambda) \|  \leq C |x|^2 , \qquad
	\| K^{\leq}_\th (x,\th,\lambda) -\th\| \leq C |x|
	\end{equation*}
	and
	\begin{equation*}
	R^{\leq}_x(x,\th,\lambda)=x-\overline a(\lambda) x^N+\O(|x|^{N+1}), \qquad
	R^{\leq}_\th(x,\th,\lambda) = \th + \omega,
	\end{equation*}
with $C>0$, and such that, in the complex domain $S(\beta_0,\rho_0)\times \TT^d_{\sigma_0} \times \Lambda_{\C}$:
	$$
	E^{\leq}=(E^{\leq}_{x},E^{\leq}_{y},E^{\leq}_\th) := F\circ K^{\leq}  -
	K^{\leq} \circ R^{\leq} = (\O(|x|^{Q+N}), \O(|x|^{Q+N}),\O(|x|^{Q+N-1})).
	$$
(We are implicitly assuming that $\beta_0,\rho_0$ are small enough so that the holomorphic extension of $F$ is well defined on
$K^{\leq} \big (S(\beta_0,\rho_0) \times \TT^d_{\sigma_0} \times \Lambda_\C \big )$.
In addition, as it is proven in Remark~\ref{RemarkR}, if $\beta_0,\rho_0$ are small enough, the composition $K^\leq \circ R^\leq $ is well defined.)

Then, for any $0<\sigma<\sigma_0$, there exist $\beta,\rho>0$, an open set $\Lambda_\C'\subset \Lambda_\C$ and a unique analytic function $\Delta$,
$$
\Delta : S(\beta,\rho) \times \TT^d_{\sigma} \times \Lambda_{\C}' \to \C^{1+m}\times \TT^d_{\sigma},\qquad \Delta=(\Delta_x,\Delta_y,\Delta_\th),
$$
satisfying
$$
\Delta_{x,y}=\O(|x|^{Q+1}),\,\qquad \Delta_\th =\O(|x|^{Q}),
$$
such that
	\begin{equation*}
	F\circ (K^{\leq} + \Delta ) = (K^{\leq} + \Delta)\circ R^{\leq}\qquad \text{ in }\quad  S(\beta,\rho)\times \TT^d_{\sigma}\times \Lambda_\C'.
	\end{equation*}
	\end{theorem}
The proof of this result is postponed to Section~\ref{sec:proofsmapstrue}.
\begin{remark}
For the sake of generality we have considered the case that $a$ and the matrix $B$ depend on both, angles $\theta$ and
parameters $\lambda$. However, in the celestial mechanics example we work with in Section~\ref{sec:N+1bodyproblem}, they are constants.
\end{remark}
The following theorem is devoted to the computation of an approximation of a solution of the
semiconjugation condition $F\circ K=K\circ R$ when $F$ is of the form \eqref{system}.
The solution is certainly not unique. We have choosen a structure for the terms that
appear in the approximation which makes it suitable for the application of Theorem~\ref{thetruemanifold}.
There is a lot of freedom for obtaining the terms of $K$ and $R$.
This freedom is seen when solving the cohomological equations at each  order.
Our main motivation has been to show that such approximation actually exists and is computable.
We refer to the reader to Section~\ref{sec:proofsmapsformal} for the computation algorithm.

\begin{theorem}[A computable approximation]\label{formalmanifold}
	Let $F$ be a real analytic map of the form \eqref{system} satisfying conditions (i)-(v).
	Assume also
	\begin{enumerate}
	\item $\omega$ is Diophantine,
	\item $\overline{a}(\lambda)\neq 0$ for $\lambda\in \Lambda$,
	\item $\overline{B}(\lambda) + j \overline{a}(\lambda)$ is invertible for $j\geq 2$ and $\lambda\in \Lambda$.
	\end{enumerate}
	Let $\U_\C \times \TT^d_\sigma \times \Lambda_\C$ be a complex domain to which $F$ can be holomorphically extended.
	
	Then, for any $j\geq 1$ there exist real analytic functions $K^{(j)}= (K_x^{(j)},K_y^{(j)},K_\theta^{(j)})$,
	$R^{(j)}=(R_x^{(j)}, R_\theta^{(j)})$ of the form
	\begin{align}
	K^{(j)}_{x}(x,\theta, \lambda)   &= x+\sum_{l=2}^j \overline K^l_{x}(\lambda) x ^l +
	\sum _{l=1}^j \widetilde K^{l+N-1}_{x}(\theta,\lambda) x^{l+N-1},  \label{formkx} \\
	K^{(j)}_{y}(x,\theta, \lambda)   &= \sum_{l=2}^j \overline K^l_{y}(\lambda) x ^l +
	\sum _{l=2}^j \widetilde K^{l+N-1}_{y}(\theta,\lambda) x^{l+N-1}, \label{formky} \\
	K^{(j)}_{\theta}(x,\theta, \lambda)  & = \th
	+\sum_{l=1}^{j-1} \overline K^l_\th (\lambda) x^l
	+ \sum_{l=1}^{j-1} \widetilde K^{l+P-1}_{\theta}(\theta,\lambda) x^{l+P-1}
	,  \label{formkth} \\
	R^{(j)}_x(x,\theta, \lambda) &=  \begin{cases}
	x-\overline a(\lambda)x^N, & 1\le j \le N-1, \\
	x-\overline a (\lambda) x^N+ b(\lambda) x^{2N-1}, & j \ge N,
	\end{cases}\notag \\
	R^{(j)}_\theta (x,\theta, \lambda) &= \theta +\omega
	+  \sum_{l=1}^{\min\{j-1,N-P\}}  R^{l+P-1}_{\theta}(\lambda) x^{l+P-1},  	\label{formrth}
	\end{align}	
	such that $	
	E^{(j)}= (E^{(j)}_x, E^{(j)}_y, E^{(j)}_\theta):= F\circ K^{(j)} -  K^{(j)} \circ R^{(j)} $
	satisfies
	\begin{equation}
	E^{(j)}_{x,y} = \O(|x|^{j+N}), \qquad E^{(j)}_\theta = \O(|x|^{j+P-1},|x|^{j+N-1}).\label{ordreE}
	\end{equation}
	Notice that, as a consequence, $K^{(j)} - K^{(j-1)}= \O(|x|^j)$.
	
	Concerning the complex domain of these functions, for any $\sigma'<\sigma$, there exists an open set $\Lambda_{\C}'\subset \Lambda_{\C}$ such that the functions
	$b(\lambda), \overline K^{l}(\lambda), R^{l+P-1}(\lambda)$ are analytic on $\Lambda_{\C}'$ and
	$\widetilde K^{l+N-1}(\theta,\lambda)$ can be holomorphically extended to $\TT^{d}_{\sigma'} \times \Lambda_{\C}'$.
\end{theorem}
\begin{remark}
Assuming that $F$ is a $\mathcal{C}^{r+1}$ map and that for all $l,k\in \N$ such that $l+k\leq r$, $F_{lk}(\th,\lambda)$
are real analytic with analytic continuation to $\TT^d_\sigma \times \Lambda_\C$, we obtain the same same result as the one
stated in Theorem~\ref{formalmanifold} for $j\leq r$. In this case the hypothesis (3) is only needed for $j\leq r$.

When $F$ is a $\mathcal{C}^{r+1}$ map, the existence of $K^{(j)}$ and $R^{(j)}$ satisfying~\eqref{ordreE} is also guaranteed
up to some value $j=r^*<r$. However, we lose regularity with respect to $\th$.
\end{remark}
Combining Theorems~\ref{thetruemanifold} and~\ref{formalmanifold} we obtain easily checkeable conditions for the existence of a
stable invariant manifold associated to the invariant torus $\mathcal{T}^d$ defined in~\eqref{torusinv}.
In Section~\ref{sec:proofcorollary} we provide the proof of the next corollary.
\begin{corollary}\label{mapcorollary}
Let $F$ be a real analytic map, having the form~\eqref{system}, satisfying conditions (i)-(v). Assume that
	\begin{enumerate}
		\item $P\geq N$,
		\item $\omega$ is Diophantine,
		\item $\overline a(\lambda)>0$ for all $\lambda\in \Lambda$,
		\item $\Re \Spec \overline B(\lambda)>0$ for all $\lambda\in \Lambda$.
	\end{enumerate}	
Let $\mathcal{U}_{\C}\times \TT^d_\sigma \times \Lambda_\C$ be the complex set where $F$ can be holomorphically extended.
Then, for any $\sigma'<\sigma$, there exist $\Lambda_{\C}'\subset \Lambda_\C$, $\beta,\rho>0$ and two real analytic functions
$$
K : S(\beta,\rho)\times \TT^d_{\sigma'}\times \Lambda_\C' \to \C^{1+m}\times \TT^d_{\sigma'},
\qquad R: S(\beta,\rho) \times \TT^d_{\sigma'} \times \Lambda_\C' \to S(\beta,\rho)\times \TT^d_{\sigma'}
$$
such that satisfy the invariance equation $F\circ K -K\circ R = 0$.

In addition they are of the form
\begin{equation}\label{formKR}
\begin{aligned}
K(x,\theta,\lambda) & =(x,0,\th + \O(|x|))+\O(|x|^{2}),\\ R(x,\theta,\lambda) &= (x- \overline{a}(\lambda) x^N + b(\lambda) x^{2N-1}, \th +\omega).
\end{aligned}
\end{equation}

Concerning regularity at $x=0$, the parameterization $K$ is $\mathcal{C}^\infty$ on $[0,\rho) \times \TT^d\times \Lambda$.

Given $\lambda\in \Lambda$, the local stable invariant set
$$
W^{\rm s}_{\rho}(\lambda) = \{(x,y,\th) \in \U\times \TT^d \,\mid \, F^k (x,y,\th,\lambda) \in \big ( B_{\rho}\times \TT^d\big ) \cap \{x>0\}\},
$$
associated to the normally parabolic invariant torus $\mathcal{T}^d$ defined in~\eqref{torusinv}, satisfies
$W^{\rm s}_{\rho}(\lambda) = K([0, \rho)\times \TT^d \times \{\lambda\})$.
\end{corollary}

The proof of this corollary is deferred to Section \ref{sec:proofcorollary}.

Applying the previous results in the case $m=0$ (that is, the map does not depend on the $y$-variable) we obtain the following
conjugation theorem:
\begin{corollary}[Conjugation result for maps] \label{conjres}
Let $F$ be a real analytic map of the form~\eqref{system}, with $m=0$, that is:
$$
F(x,\theta,\lambda) = (x-a(\th,\lambda)x^N + f_{\geq N+1}(x,\th,\lambda)  , \th + \omega +  h_{\geq P}(x,\th,\lambda))
$$
being $f_{\geq N}= f_N + f_{\geq N+1}$, $h_{\geq P} = h_P+ h_{\geq P+1}$ satisfying the corresponding conditions given (i)-(v).
Assume that
\begin{enumerate}
\item $P\geq N$,
\item $\omega$ is Diophantine,
\item $\overline{a}(\lambda)>0$ for $\lambda\in \Lambda$.
\end{enumerate}
Let  $\U_\C \times \TT^d_\sigma \times \Lambda_\C$ be such that $F$ can be analytically extended to it.
Then for any $\sigma'<\sigma$ there exist $\beta,\rho>0$, an open set $\Lambda_\C' \subset \Lambda_\C$
and a real analytic function $b:\Lambda_\C' \to \C$ such that the map $F$ is analytically conjugated to
$$
R(x,\theta,\lambda) = (x-\overline{a}(\lambda) x^N + b(\lambda) x^{2N-1} , \th + \omega),
$$
on $S(\beta,\rho) \times \TT^d_{\sigma'}$ for any $\lambda \Lambda_\C'$.

In addition the conjugation is $\mathcal{C}^{\infty}$ on $[0,\rho)\times \TT^d \times \Lambda$.
\end{corollary}
This conjugation result extends some of the results by Takens \cite{Takens} and Voronin \cite{Voronin}
to parabolic tori.

\subsection{Results for flows}\label{subsec:flowcase}

We consider a autonomous vector field $X(x,y,\theta,t,\lambda)$ depending quasi periodically on time, having
the form
\begin{equation}\label{vectorfield_truemanifold}
\begin{aligned}
\dot{x} &= -a(\theta,t,\lambda) x^N  + f_{N}(x,y,\theta,t,\lambda) + f_{\geq N+1}(x,y,\theta,t,\lambda) \\
\dot{y} &= x^{N-1} B(\theta,t,\lambda) y + g_{N}(x,y,\theta,t,\lambda) + g_{\geq N+1}(x,y,\theta,t,\lambda) \\
\dot{\theta} &= \omega + h_P(x,y,\theta,t,\lambda) + h_{\geq P+1}(x,y,\theta,t,\lambda),
\end{aligned}
\end{equation}
with $(x,y)\in \mathbb{R}^{1+m}$, $\theta \in \mathbb{T}^d$ and $\lambda\in \Lambda$.
The functions involved in the definition of the vector field $X$, i.e. $a,B,f_{N},g_{N},h_{P}, f_{\geq N+1}, g_{\geq N+1}, h_{\geq P+1}$ and the numbers $N,P,\omega$,
satisfy the same conditions as the ones imposed to the functions involved in the case of maps in Section~\ref{subsec:mapcase} (see conditions (i)-(v)
below~\eqref{system}). The periodic and autonomous case are included as a particular case when $d'=1$ and $d'=0$ respectively.

As in the map case, the torus $\mathcal{T}^d=\{0\}\times \{0\} \times \TT^d$ is an invariant object such that all its normal directions are parabolic.
Again, we look for invariant manifolds associated to it by means of the parameterization method. We emphasize that, in the flow case, we look for
$K(x,\theta,t,\lambda)$ and a vector field $Y(x,t,\th,\lambda)$ such that they satisfy the invariance condition
$$
X(K(x,\theta,t,\lambda),t,\lambda) - D K(x,\theta,t,\lambda) Y(x,t,\theta,\lambda)-\partial_t K(x,\theta,t,\lambda) =0, \qquad D=\partial_{(x,\theta)}.
$$

The following \textit{a posteriori} result is proven in Section~\ref{sec:proofedostrue}.
\begin{theorem}[\textit{A posteriori} result] \label{main_theorem_flow_case}
Let $X$ be a real analytic vector field, having the form~\eqref{vectorfield_truemanifold}, satisfying conditions (i)-(v).

Let $\nu\in \R^{d'}$ be the time frequencies (see Section~\ref{subsec:notation}) of $X$. If $X$ is an autonomous vector field, $d'=0$.
Assume that
	\begin{enumerate}
		\item $P\geq N$,
		\item either $(\omega,\nu)=(\omega_1,\cdots,\omega_d,\nu_1,\cdots,\nu_{d'})$ is Diophantine or the functions $a,B$ depend neither on $\theta$ nor on $t$.
		\item $\overline a(\lambda)>0$ for $\lambda\in \Lambda$,
		\item $\Re \Spec \overline B(\lambda)>0$ for $\lambda\in \Lambda$.
	\end{enumerate}
Let $Q\geq N$ and assume that, for some $\beta_0,\rho_0,\sigma_0>0$ and
$\Lambda_{\C}\subset \C^p$, there exist
$K^{\leq} : S(\beta_0,\rho_0)\times \TT^{d}_{\sigma_0} \times \CS_{\sigma_0} \times \Lambda_{\C} \to \C^{1+m} \times \TT^d_{\sigma_0}$ and
$Y^{\leq} : S(\beta_0,\rho_0)\times \TT^d_{\sigma_0}\times \CS_{\sigma_0}\times \Lambda_{\C}\to \C \times \TT^d_{\sigma_0}
$
depending quasi periodically on $t$ with the same frequencies than $X$, satisfying
\begin{align*}
|K^{\leq}_x(x,\theta,t,\lambda)&-x|\leq C |x|^N ,\quad \|K^{\leq}_y(x,\theta,t,\lambda)\|\leq C |x|^2,
\quad \|K^{\leq}_{\th}(x,\theta,t,\lambda)-\th\|\leq C|x|,  \\
&Y^{\leq}_x(x,\th,t,\lambda) = x- \overline{a}x^{N} + \mathcal{O}(x^{N+1}),\qquad
Y^{\leq}_{\th} (x,\th,t,\lambda) = \omega
\end{align*}
for some constant $C$ and such that in the complex domain $S(\beta_0,\rho_0)\times \TT_\sigma^{d}\times \CS_{\sigma} \times \Lambda_\C$,
satisfies
\begin{equation}\label{condEleqfluxos}
E^{\leq} := X \circ K^{\leq} -
DK^{\leq} Y^{\leq}-\partial_t K^{\leq} = (\mathcal{O}(|x|^{Q+N}),\mathcal{O}(|x|^{Q+N}),\mathcal{O}(|x|^{Q+N-1})).
\end{equation}

Then, for any $\sigma< \sigma_0$, there exist $\beta,\rho>0$, an open set $\Lambda_\C'\subset \Lambda_\C$ and a unique analytic function $\Delta$
$$
\Delta : S(\beta,\rho) \times \TT^{d}_{\sigma} \times \CS_{\sigma} \times \Lambda_{\C}' \to \C^{1+m}\times \TT^d_{\sigma},
\qquad \Delta=(\Delta_x,\Delta_y,\Delta_\th),
$$
satisfying
$$
\Delta_{x,y}=\O(|x|^{Q+1}),\,\qquad \Delta_\th =\O(|x|^{Q})
$$
and
$$
X\circ (K^{\leq}+\Delta) - \big (DK^{\leq}+\Delta\big ) Y^{\leq}-\partial_t (K^{\leq}+\Delta )=0
,\qquad \text{ in } S(\beta,\rho)\times \TT^{d}_{\sigma} \times \CS_{\sigma}\times \Lambda_\C'.
$$
Writing $K=K^{\leq} + \Delta$ the infinitesimal invariance equation is equivalent to
$$
\Phi(t;s,K(x,\theta,s,\lambda),\lambda) = K(\psi(t;s,x,\theta,\lambda),t,\lambda)
$$
with $\Phi(t;s,x,y,\theta,\lambda)$ and $\psi(t;s,x,\theta,\lambda)$ the flow of $X$ and $Y^\leq$ respectively.

Finally, if the vector field $X$ is autonomous, that is $d'=0$, and the approximated parameterization $K^{\leq}$
does not depend on $t$, then $\Delta $ is also independent of $t$.
\end{theorem}

As we did for the case of real analytic maps, we provide below a direct algorithm to compute
an approximation $K^{\leq}$ and a vector field $Y^\leq$ satisfying~\eqref{condEleqfluxos}.
The following result gives the form of these functions. In addition,
an algorithm to compute them is provided in Section~\ref{sec:proofsedosformal}.

\begin{theorem}[A computable approximation] \label{formal_part_theorem_flows}
	Let $X$ be a real analytic vector field
	of the form~\eqref{vectorfield_truemanifold} satisfying conditions (i)-(v), with
	holomorphic continuation to $\U_\C\times \TT^{d}_{\sigma}\times \CS_{\sigma}\times \Lambda_{\C}$ for
	some $\sigma>0$.
	Assume in addition that
	\begin{enumerate}
	\item $(\omega,\nu)$ is Diophantine,
	\item $\overline{a}(\lambda)>0$ for $\lambda\in \Lambda$,
	\item $\overline{B}(\lambda)+j\overline{a}(\lambda)$ is invertible for $j\geq 2$ and $\lambda \in \Lambda$.
	\end{enumerate}
	
	Let $\nu\in \R^{d'}$ be the time frequencies.
	Then, for any $j\geq 1$ there exist a real analytic function
	$K^{(j)}= (K_x^{(j)},K_y^{(j)},K_\theta^{(j)})$,
	and a real analytic vector field
	$Y^{(j)}=(Y_x^{(j)}, Y_\theta^{(j)})$, depending quasi periodically on $t$ with frequency $\nu$, of the form
	\begin{align}
	K^{(j)}_{x}(x,\theta,t, \lambda)   &= x+\sum_{l=2}^j \overline K^l_{x} x ^l +
	\sum _{l=1}^j \widetilde K^{l+N-1}_{x}(\theta,t) x^{l+N-1},  \label{formkvfx} \\
	K^{(j)}_{y}(x,\theta,t, \lambda)   &= \sum_{l=2}^j \overline K^l_{y} x ^l +
	\sum _{l=2}^j \widetilde K^{l+N-1}_{y}(\theta,t) x^{l+N-1}, \label{formkvfy} \\
	K^{(j)}_{\theta}(x,\theta,t, \lambda)  & = \th
	+\sum_{l=1}^{j-1} \overline K^l_\th x ^l
	+ \sum_{l=1}^{j-1} \widetilde K^{l+P-1}_{\theta}(\theta,t) x ^{l+P-1}
	,  \label{formkvfth} \\
		Y^{(j)}_x(x,\theta,t, \lambda) &=  \begin{cases}
	-\overline a(\lambda)x^N & 1\le j \le N-1, \\
	-\overline a (\lambda) x^N+ b(\lambda) x^{2N-1} & j \ge N,
	\end{cases}\label{formrvfx}
	\\
	Y^{(j)}_\theta (x,\theta, t,\lambda) &= \omega
	+  \sum_{l=1}^{\min\{j-1,N-P\}}  Y^{l+P-1}_{\theta}(\lambda) x^{l+P-1},  	\label{formrvfth}
	\end{align}	
	such that $	
	E^{(j)}= (E^{(j)}_x, E^{(j)}_y, E^{(j)}_\theta)^\top:=
	X\circ K^{(j)} - D K^{(j)}   Y^{(j)} - \partial_t K^{(j)}
	$
	satisfies
	\begin{equation}
	E^{(j)}_{x,y} = \O(|x|^{j+N}), \qquad E^{(j)}_\theta = \big (\O(|x|^{j+P-1}),\O(|x|^{j+N-1})).\label{ordreEvf}
	\end{equation}
	Notice that, as a consequence, $K^{(j)}-K^{(j-1)}=\O(|x|^{j})$ and $Y^{(j)}$ does not depend on $(\theta,t)$.
	
	Concerning the complex domain, for any $0<\sigma'<\sigma$
	there exists an open set $\Lambda_{\C}'\subset \Lambda_{\C}$
	such that for any $\sigma'<\sigma$, all the functions can be
	holomorphically extended to either $\Lambda_\C'$ or $\TT^{d}_{\sigma'} \times \CS_{\sigma'}\times \Lambda_\C'$.
	
	In addition, when the vector field $X$ is autonomous, we can choose $K^{(j)}$ independent on $t$.
\end{theorem}
\begin{remark}
Assuming that $X$ is a $\mathcal{C}^{r+1}$ vector field of the form~\eqref{vectorfield_truemanifold} and that for $l,k\in \N$
	such that $l+k\leq r$, $X_{l,k}(\th,t,\lambda)$ are real analytic with
	analytic continuation to $\TT^{d}_{\sigma}\times \CS_{\sigma}\times \Lambda_{\C}$ for
	some $\sigma>0$ the same result as the one stated in the previous theorem can be proven.
\end{remark}
\begin{remark}
We can treat $t$ as a new angle by adding the equation $\dot{t}=1$. This means to deal with the frequency vector $(\omega,1)$. However
we maintain $\th$ and $t$ separate to find formulas directly applicable to the examples.
\end{remark}
The existence of a parabolic stable manifold for a vector field having the form~\eqref{vectorfield_truemanifold} is a direct application of
the previous results.
\begin{corollary}
Let $X$ be a real analytic vector field, depending quasiperiodically in time, having the form~\eqref{vectorfield_truemanifold}
and satisfying conditions (i)-(v).
Let $\nu\in \R^{d'}$ be the time frequency vector. Assume that
	\begin{enumerate}
		\item $P\geq N$,
		\item $(\omega,\nu)$ is Diophantine,
		\item $\overline a(\lambda)>0$ for $\lambda\in \Lambda$,
		\item $\Re \Spec \overline B(\lambda)>0$ for $\lambda\in \Lambda$.
	\end{enumerate}	
Let $\mathcal{U}_{\C}\times \TT^{d}_{\sigma}\times \CS_{\sigma} \times \Lambda_\C$ be the complex set where $X$ can be holomorphically extended.
Then, for any $\sigma'<\sigma$, there exist an open set $\Lambda_{\C}'\subset \Lambda_\C$, $\beta,\rho>0$ and two real analytic functions such that
$$
K : S(\beta,\rho)\times \TT^{d}_{\sigma'} \times \CS_{\sigma'}\times \Lambda_\C' \to \C^{1+m}\times \TT^d_{\sigma'},
\qquad Y: S(\beta,\rho) \times \Lambda_\C' \to S(\beta,\rho)\times \TT^d_{\sigma'}
$$
and they satisfy the invariance equation $X(K,t,\lambda) - D K\cdot Y-\partial_t K =0$,
with $D=\partial_{x,\th}$. In the autonomous case, both $K$ and $Y$ are independent of $t$.

Moreover:
$$
K(x,\theta,t,\lambda)=(x,0,\th + \O(|x|))+\O(|x|^{2}),\quad Y(x,\lambda) = (- \overline{a}(\lambda) x^N + b(\lambda) x^{2N-1}, \omega).
$$

Concerning the regularity at $x=0$, the parameterization $K$ is $\mathcal{C}^\infty$ on $[0,\rho) \times \TT^{d}\times \R \times \Lambda$.

Let $\lambda\in \Lambda$. The local stable invariant set
$$
W^{\rm s}_{\rho}(\lambda) = \{(x,y,\th) \in \U\times \TT^d \,:\, F^k (x,y,\th,\lambda) \in \big (B_{\rho}\times \TT^d\big ) \cap \{x>0\}\}
$$
associate to the normally parabolic invariant torus $\{0\}\times \{0\} \times \TT^d$ satisfies
$W^{\rm s}_{\rho}(\lambda) = K([0, \rho)\times \TT^d \times \{\lambda\})$.
\end{corollary}
The proof of this corollary is completely analogous to the proof of Corollary~\ref{mapcorollary}. To finish we present a conjugation result analogous to Corollary \ref{conjres}.
\begin{corollary}[Conjugation result for flows]
Let $X$ be a real analytic vector field of the form~\eqref{vectorfield_truemanifold}
and satisfying conditions (i)-(v) with $m=0$, that is we impose $X$ to be as:
$$
X(x,\theta,t,\lambda) = (-a(\th,t,\lambda)x^N + f_{\geq N+1}(x,\th,t,\lambda)  ,  \omega +  h_{\geq P}(x,\th,t,\lambda))
$$
being $h_{\geq P} = h_P+ h_{\geq P+1}$.
Assume that
\begin{enumerate}
\item $P\geq N$,
\item $\omega$ is Diophantine,
\item $\overline{a}(\lambda)>0$ for $\lambda\in \Lambda$.
\end{enumerate}
Let  $\U_\C \times \TT^{d}_{\sigma} \times \CS_{\sigma} \times \Lambda_\C$ be such that $X$ can be analytically extended to it.
Then for any $\sigma'<\sigma$ there exist $\beta,\rho>0$, an open set $\Lambda_\C' \subset \Lambda_\C$
and a real analytic function $b:\Lambda_\C' \to \C$ such that the vector field $X$ is analytically conjugated to
$$
Y(x,\lambda) = (-\overline{a}(\lambda) x^N + b(\lambda) x^{2N-1} , \omega), \qquad (x,\lambda) \in S(\beta,\rho) \times \Lambda_\C'
$$
with the conjugation map defined on $S(\beta,\rho) \times \TT^{d}_{\sigma'} \times \CS_{\sigma'} \times \Lambda_{\C}'$.

In addition the conjugation is $\mathcal{C}^{\infty}$ on $[0,\rho)\times \R\times \Lambda$.
\end{corollary}

%% file: BFM_Tori_Celestial_mechanics-v3.tex
\section{Invariant manifolds of infinity in the planar $(n+1)$-body problem}\label{sec:N+1bodyproblem}

In this section we present two examples from celestial mechanics where it is possible
to apply our results to obtain families of Diophantine parabolic tori. These families
lie in cylinders, and the invariant manifolds of the parabolic tori give rise to the
invariant manifolds of these ``normally parabolic'' cylinders.

\subsection{The restricted planar $(n+1)$-body problem}

The restricted $(n+1)$-body problem models the motion of a massless body under the Newtonian
gravitational attraction of $n$ bodies, the primaries, with masses $m_j$, $j=1,\dots, n$, which evolve under their mutual gravitational attraction. It can be seen as  the limit of the $(n+1)$-body problem when the mass of one the bodies is taken $0$.
The problem is planar when the motion of all the bodies is confined
in a plane.

Here we assume that the primaries move in a quasiperiodic motion, that is, their positions in the plane in some inertial reference system are given by $q_j(\omega t)$ where
\[
q_j : \TT^d \to \R^2, \qquad j=1,\dots,n.
\]
We will assume that $\omega \in \R^d$ is Diophantine.
Such motions do exist (see Section~\ref{sec:thenp1bodyproblem}). The functions $q_j$ are analytic in a complex strip. By the conservation of the linear momentum, we can assume that
\[
\sum_{j=1}^n m_j q_j(\omega t) = 0, \qquad t \in \R.
\]

Let $q \in \R^2$ be the position of the massless body in the current reference system.
Then, taking the unit of time in which the universal gravitational constant becomes~$1$, the restricted planar $(n+1)$-body problem is Hamiltonian
with Hamiltonian function
\[
H(q,p,t) = \frac{1}{2} \|p\|^2 - U(q,t), \qquad (q,p) \in \R^2\times \R^2,
\]
where
\[
U(q,t) = \sum_{j=1}^n \frac{m_j}{\|q-q_j(\omega t)\|}.
\]
It has $2+d$ degrees of freedom.

Taking polar coordinates in the plane, $q = r e^{i\theta}$, with conjugate momenta
$p = y e^{i\theta} + i G e^{i\theta}/ r$, the Hamiltonian (we use the same letter to denote it) becomes
\[
H(r,\theta,y,G,t) = \frac{1}{2} \left( y^2+ \frac{G^2}{r^2}\right) - V(r,\theta,t),
\]
where
\[
V(r,\theta,t) = U (r e^{i\theta}, t) = \sum_{j=1}^n \frac{m_j}{|r e^{i\theta}-q_j(\omega t)|}.
\]
If we assume that $r \gg q_j$ and use that $m_1 q_1 + \cdots + m_N q_N \equiv 0$,
\[
V(r,\theta,t) = \frac{1}{r} \sum_{j=1}^n \frac{m_j}{|1- e^{-i\theta} r^{-1} q_j(\omega t)|} = \frac{\sum_{j=1}^n m_j}{r} + \O\left(\frac{1}{r^3}\right),
\]
where the remainder $\O(r^{-3})$ depends on $(r,\theta,t)$, quasiperiodically on $t$.

Let $M = m_1+\cdots + m_n$. We consider new variables by setting $r = 2/x^2$ (McGehee coordinates).
This change of variables transforms the $2$-form $dr \wedge dy + d\theta \wedge dG$ into
\[
-4x^{-3} \,dx \wedge dy + d\theta \wedge d G.
\]
This means that the equations of motion for the Hamiltonian in the new variables
\begin{equation*}
\widetilde H(x,\theta,y,G) = H(2/x^2,\theta,y,G) =  y^2/2+ G^2 x^4/4 + M x^2/2 + \O(x^6)
\end{equation*}
are
\[
\dot x  = - \frac{x^3}{4} \frac{\partial \widetilde H}{\partial y}, \qquad \dot y   = - \frac{x^3}{4} \left(-\frac{\partial \widetilde H}{\partial x}\right),
\qquad \dot \theta  = \frac{\partial \widetilde H}{\partial G}, \qquad  \dot G   = -\frac{\partial \widetilde H}{\partial \theta}.
\]
Since the term $\O(x^6)$ is a function of $(x,\theta,t)$,  the equations of motion are
\begin{equation}
\label{eqofmotionrestricted}
\dot x  =  - \frac{1}{4} x^3 y, \qquad \dot y    =  - \frac{M}{4} x^4 + \O(x^6),\qquad \dot \theta   =  \frac{1}{2} G x^4,\qquad \dot G  =  \O(x^6).
\end{equation}
It is clear from the above equations that, for any $(\theta_0,G_0) \in \TT\times \R$, the set
\begin{equation*}
\T_{\theta_0,G_0} = \{x=0, y=0,\theta = \theta_0,
G = G_0\} \cong \TT^d
\end{equation*}
is an invariant torus of the system with frequency vector $\omega$.
\marginpar{falta completar prop}
\begin{proposition}
For each $(\theta_0,G_0) \in \TT\times \R$, $\T_{\theta_0,G_0}$ is a Diophantine parabolic torus
of $\widetilde H$ with parabolic unstable and stable invariant manifolds $W^{u,s}$ which
admit $C^{\infty}$ parametrizations
\[
K^{u,s} :[0,\delta) \times \TT^d \to \R^2
\]
analytic in a complex domain of the form $S(\beta,\delta) \times \TT_{\sigma}^d \supset (0,\delta) \times \TT^d$.
\end{proposition}

\begin{proof}
Scaling $x$ and $y$ and introducing the new angle $\alpha = \theta+Gy$, equations~\eqref{eqofmotionrestricted} become
\begin{equation}
\label{eqofmotionrestricted_alpha}
\dot x  =   - \frac{1}{4} x^3 y ,\qquad \dot y   =  - \frac{1}{4} x^4 + \O(x^6),\qquad
\dot \alpha  = \O(x^6) ,\qquad  \dot G   =  \O(x^6).
\end{equation}
Notice that, if we disregard the $(\theta,G)$ variables, $y=\pm x$ are characteristic directions of the system above.
For this reason, we consider new variables $u = (x-y)/2$, $v = (x+y)/2$. Now, defining $z = ( \theta,G)$, for any $z_0 = (\theta_0,G_0) \in \TT\times \R$, we consider the new variables $\tilde z = (\tilde z, \tilde G) = (z-z_0)/(u+v)$. In order to apply Theorem~\ref{main_theorem_flow_case}, we also introduce $\phi = \omega t$. Summarizing, in these new variables, system~\eqref{eqofmotionrestricted_alpha} becomes
\[
\begin{aligned}
\dot u & =   \frac{1}{4} (u+v)^3 \left(u+\O((u+v)^3)\right), \\
\dot v  & =  - \frac{1}{4} (u+v)^3 \left(v+\O((u+v)^3)\right) , \\
\dot{\tilde{z}} & = \frac{1}{4} (u+v)^2(u-v) \tilde z + \O((u+v)^5), \\
\dot \phi & = \omega,
\end{aligned}
\]
which satisfy the hypotheses of Theorem~\ref{main_theorem_flow_case} with $a= 1/4$, $N=6$ and any $P$.
\end{proof}

\subsection{The planar $(n+1)$-body problem}
\label{sec:thenp1bodyproblem}
Consider $n+1$ point masses, $m_i$, $i = 0,\dots, n$,
evolving in the plane under their mutual Newtonian gravitational attraction. Let $q_i\in \R^2$, $i=0,\dots,n$, be their coordinates in an inertial frame of reference.
Taking the unit of time in which the universal gravitational constant becomes~$1$, the equations of motion are
\begin{equation}
\label{equations_of_montion_n_body_problem}
m_i \ddot{q_i} = \sum_{j=0, i\neq j}^n m_i m_j \frac{q_j-q_i}{\|q_j-q_i\|^3} = \frac{\partial U}{\partial q_i}(q_0,\dots,q_n), \qquad i =0,\dots, n,
\end{equation}
where
\begin{equation*}
U(q_0,\dots,q_n) = \sum_{0\le i < j \le n} \frac{m_i m_j}{\|q_j-q_i\|}.
\end{equation*}
Introducing the momenta $p_i = m_i \dot q_i$, $i=0,\dots,n$, and the kinetic energy
\begin{equation*}
T(p_0,\dots,p_n) = \sum_{i=0}^n \frac{1}{2m_i} p_i^2,
\end{equation*}
system~\eqref{equations_of_montion_n_body_problem} is Hamiltonian with $2(n+1)$ degrees of freedom and Hamiltonian function
$
H(q,p) = T(p)-U(q),
$
that is, \eqref{equations_of_montion_n_body_problem} becomes
\[
\dot q_i = \frac{\partial H}{\partial p_i}, \qquad \dot p_i = -\frac{\partial H}{\partial q_i}, \qquad i = 0,\dots,n.
\]
The $(n+1)$-body problem has several well known first integrals besides the energy: the total linear momentum, $p_0+\cdots +p_n$, and the total angular momentum, $\det(q_0,p_0) + \cdots + \det(q_n,p_n)$.
Here it will be convenient to reduce the linear momentum. To do so, we consider the \emph{Jacobi coordinates}, $(\tilde q, \tilde p)$. This set of coordinates is defined as follows:  the position of the $j$-th body is measured with respect to the center of mass of the bodies $0$ to $j-1$. Since they are a linear combination of the original variables, the momenta are also changed through a linear map. The new coordinates satisfy
\begin{equation*}
\begin{aligned}
\tilde q_0 & = q_0 \\
\tilde q_j & = q_j - \frac{1}{M_j}\sum_{0\le \ell \le j-1} m_{\ell} q_{\ell}, \qquad j=1,\dots, n,
\end{aligned}
\end{equation*}
where $M_j = \sum_{\ell=0}^{j-1} m_{\ell}$, $j \ge 1$, with conjugate momenta
\begin{equation*}
\begin{aligned}
\tilde p_j & = p_j + \frac{m_j}{M_{j+1}}\sum_{j+1\le \ell \le n} p_{\ell}, \qquad j=0,\dots, n-1, \\
\tilde p_n & = p_n.
\end{aligned}
\end{equation*}
Once the transformation of the momenta is found, the inverse of the change is determined\footnote{Indeed, the linear change of variables $(\tilde q, \tilde p) = (Aq,Bp)$ is symplectic if and only if $A^{\top}B = \Id$.}. It is given by
\begin{equation*}
\begin{aligned}
q_0 & = \tilde q_0 \\
q_j & = \tilde q_j +\sum_{0\le \ell \le j-1} \frac{m_{\ell}}{M_{\ell+1}} \tilde q_{\ell}, \qquad j=1,\dots, n.
\end{aligned}
\end{equation*}
Now we make the reduction of the total linear momentum. In the new variables, this first integral is $\tilde p_0$, which implies that the Hamiltonian does not depend on $\tilde q_0$.
We can assume $\tilde p_0 = 0$. Then, it is easy to check that\footnote{The kinetic energy part of the Hamiltonian, in the new variables, is
\[
T(B^{-1} \tilde p) = \frac{1}{2}\tilde p^{\top}(B^{-1})^{\top} M B^{-1} \tilde p =
\frac{1}{2}\tilde p^{\top}A M A^{\top} \tilde p,
\]
where $M= \mathrm{diag}\,(1/m_0,\dots,1/m_n)$. When $\tilde p_0 = 0$, the above expression is diagonal.}, in the new variables, the Hamiltonian becomes
\begin{equation}
\label{np1_body_problem_in_Jacobi_coordinates}
\widetilde H(\tilde q_1,\dots, \tilde q_n, \tilde p_1,\dots,\tilde p_n) = \sum_{1\le j \le n}\frac{1}{2\mu_j} \|\tilde p_j\|^2 - \widetilde U (\tilde q_1,\dots, \tilde q_n),
\end{equation}
where $1/\mu_j = 1/M_j+1/m_j$ and
\begin{equation*}
\widetilde U (\tilde q_1,\dots, \tilde q_n) = \sum_{1\le j \le n} \frac{m_0 m_j}{\left\|\tilde q_j +\sum_{1\le \ell\le j-1}\frac{m_{\ell}}{M_{\ell+1}} \tilde q_{\ell}\right\|}+
\sum_{1\le k<j \le n} \frac{m_k m_j}{\left\|\tilde q_j +\sum_{k\le \ell\le j-1}\frac{m_{\ell}}{M_{\ell+1}} \tilde q_{\ell}-\tilde q_k\right\|}.
\end{equation*}
It has $2n$-degrees of freedom.


In the following discussion it will be convenient to consider polar coordinates in the plane for each of the bodies. Let $(r_j,\theta_j)$ be defined by $\tilde q_j = r_j e^{i\theta_j}$, $j=1,\dots,n$, (identifying $\R^2$ with the complex plane in the usual way). Their conjugate momenta, $(y_j,G_j)$,  are given by $\tilde p_j = y_j e^{i\theta_j} + i \frac{G_j}{r_j}e^{i\theta_j}$ and satisfy
\[
|\tilde p_j|^2 = y_j^2+ \frac{G_j^2}{r_j^2}.
\]
In these coordinates, denoting $r=(r_1,\dots,r_n)$ and, analogously, $\theta$, $y$, $G$, the Hamiltonian $\widetilde H$
in~\eqref{np1_body_problem_in_Jacobi_coordinates} becomes
\begin{equation*}
\widehat H(r,\theta,y,G) = \sum_{j=1}^n \frac{1}{2\mu_i}\left(y_j^2+\frac{G_j^2}{r_j^2}\right)-V(r,\theta),
\end{equation*}
where
\begin{equation*}
\begin{aligned}
V(r,\theta)  = & \widetilde U(r_1 e^{i \theta_1},\dots,r_n e^{i \theta_n}) \\
 = & \sum_{1\le j \le n} \frac{m_0 m_j}{\left|r_je^{i\theta_j} +\sum_{1\le \ell\le j-1}\frac{m_{\ell}}{M_{\ell+1}} r_{\ell}e^{i\theta_{\ell}}\right|}
 \\
 &
 +
\sum_{1\le k<j \le n} \frac{m_k m_j}{\left|r_je^{i\theta_j} +\sum_{k\le \ell\le j-1}\frac{m_{\ell}}{M_{\ell+1}} r_{\ell}e^{i\theta_{\ell}}- r_k e^{i\theta_k}\right|}.
\end{aligned}
\end{equation*}
We split this potential as follows,
$
V(r,\theta) = V_0(r,\theta) +  V_1 (\hat r,  \hat \theta)
$
where $(\hat r, \hat\theta) = ( r_1,\dots, r_{n-1}, \theta_1,\dots,\theta_{n-1})$ and
\begin{equation*}
\begin{aligned}
V_0(r,\theta)  = & \frac{m_0 m_n}{\left|r_n e^{i\theta_n} +\sum_{1\le \ell\le n-1}\frac{m_{\ell}}{M_{\ell+1}} r_{\ell}e^{i\theta_{\ell}}\right|} \\
& +  \sum_{1\le k \le n-1} \frac{m_k m_n}{\left|r_n e^{i\theta_n} +\sum_{k\le \ell\le n-1}\frac{m_{\ell}}{M_{\ell+1}} r_{\ell}e^{i\theta_{\ell}}- r_k e^{i\theta_k}\right|},\\
V_1 (\hat r, \hat\theta)  = & \sum_{1\le j \le n-1} \frac{m_0 m_j}{\left|r_je^{i\theta_j} +\sum_{1\le \ell\le j-1}\frac{m_{\ell}}{M_{\ell+1}} r_{\ell}e^{i\theta_{\ell}}\right|} \\
 & +
\sum_{1\le k<j \le n-1} \frac{m_k m_j}{\left|r_je^{i\theta_j} +\sum_{k\le \ell\le j-1}\frac{m_{\ell}}{M_{\ell+1}} r_{\ell}e^{i\theta_{\ell}}- r_k e^{i\theta_k}\right|}.
\end{aligned}
\end{equation*}
We emphasize that $V_1$ does not depend on the variables $(r_n,\theta_n)$ (that is, does not depend on the last body).

We will assume that we are in a region of the phase space where $r_n \gg r_j$,  while
$r_j = \O(1)$, $j =1, \dots, n-1$. Under this assumption,
and using that
\[
\frac{1}{|1-z|}  = \frac{1}{(1-z)^{1/2}(1-\bar{z})^{1/2}} = \sum_{\ell \ge 0} c_{\ell} z^{\ell}
\sum_{k \ge 0} c_{k} \bar{z}^{k} = \sum_{\ell,k \ge 0} c_{\ell}c_k z^{\ell} \bar{z}^k,
\]
where $c_0 = 1$ and $c_1=1/2$, we have that
\[
\begin{aligned}
V_0(r,\theta) = & \frac{m_0 m_n}{r_n \left|1 +\sum_{1\le \ell\le n-1}\frac{m_{\ell}}{M_{\ell+1}} \frac{r_{\ell}}{r_n}e^{i(\theta_{\ell}-\theta_n)}\right|} \\
& +  \sum_{1\le k \le n-1} \frac{m_k m_n}{r_n \left|1 +\sum_{k\le \ell\le n-1}\frac{m_{\ell}}{M_{\ell+1}}
\frac{r_{\ell}}{r_n}e^{i(\theta_{\ell}-\theta_n)}-\frac{r_k}{r_n} e^{i(\theta_k-\theta_n)}\right|}\\
= & \frac{m_n M_n}{r_n} - \frac{m_0 m_n}{2} \sum_{1\le \ell\le n-1}\frac{m_{\ell}}{M_{\ell+1}}
\frac{r_{\ell}}{r_n^2}\left(e^{i(\theta_{\ell}-\theta_n)}+e^{-i(\theta_{\ell}-\theta_n)}\right) \\
& - \frac{m_n}{2}\sum_{1\le k \le n-1} m_k \left(\sum_{k\le \ell\le n-1}\frac{m_{\ell}}{M_{\ell+1}}
\frac{r_{\ell}}{r_n^2}\left(e^{i(\theta_{\ell}-\theta_n)} +e^{-i(\theta_{\ell}-\theta_n)}\right) \right . \\
& \left .- \frac{r_k}{r_n^2} \left(e^{i(\theta_k-\theta_n)}+e^{-i(\theta_k-\theta_n)}\right) \right) + \O\left(\frac{1}{r_n^3}\right) \\
= & \frac{m_n M_n}{r_n} + \O\left(\frac{1}{r_n^3}\right).
\end{aligned}
\]
Since we will be interested in the behaviour of the system around $r_n = \infty$ we introduce the McGehee coordinates
\[
r_n = \frac{2}{x_n^2}.
\]
The canonical form $\omega = \sum_{j=1}^n (dr_j\wedge dy_j+d\theta_j\wedge dG_j)$ becomes
\begin{equation}
\label{form_McGehee}
\omega =  \sum_{j=1}^{n-1} (dr_j\wedge dy_j+d\theta_j\wedge dG_j) -\frac{4}{x_n^3}\, dx_n \wedge dy_n + d\theta_n\wedge dG_n,
\end{equation}
that is, defining the potential
\[
\U (\hat r,x_n,\hat \theta,\theta_n) = V (\hat r,2/x_n^2,\hat \theta,\theta_n),
\]
where $(\hat r,\hat \theta) = ( r_1,\dots, r_{n-1},\theta_1,\dots, \theta_{n-1})$, and the Hamiltonian
\[
\HH (\hat r,x_n, \theta,y,G) = \widetilde H(\hat r,2/x_n^2, \theta,y,G) =
\sum_{j=1}^n\frac{1}{2 \mu_j}\left(y_j^2+ \frac{G_j^2}{r_j^2}\right)- \U (\hat r,x_n,\hat \theta,\theta_n),
\]
the equations of motion are
\[
\begin{aligned}
\dot r_j & =   \frac{\partial \HH}{\partial y_j}, & \dot y_j & = -  \frac{\partial \HH}{\partial r_j},
&\dot \theta_j & =  \frac{\partial \HH}{\partial G_j}, & \dot G_j &= -  \frac{\partial \HH}{\partial \theta_j}, \\
\dot x_n & = -\frac{x_n^3}{4}  \frac{\partial \HH}{\partial y_n}, & \dot y_n & = -\frac{x_n^3}{4} \left(- \frac{\partial \HH}{\partial x_n} \right),
&\dot \theta_n & =  \frac{\partial \HH}{\partial G_n}, & \dot G_n &= -  \frac{\partial \HH}{\partial \theta_n},
\end{aligned}
\]
where $j =1,\dots, n-1$.

Writing $\U = \U_0 + \U_1$, where
\[
\begin{aligned}
\U_0 (\hat r,x_n, \hat \theta, \theta_n) & = V_0(\hat r, 2/x_n^2, \hat \theta,\theta_n)
= \frac{m_n M_n}{2}x_n^2 + \O\left(x_n^6\right),\\
\U_1(\hat r, \hat \theta) & = V_1(\hat r,\hat \theta),
\end{aligned}
\]
then
\begin{equation}
\label{HH-HH_0-HH_1}
\HH(\hat r,x_n, \hat \theta,\theta_n,y,G) = \HH_0(\hat r,x_n, \hat \theta,\theta_n,y,G) + \HH_1(\hat r, \hat \theta, \hat y, \hat G),
\end{equation}
where $(\hat y, \hat G) = (y_1,\dots, y_{n-1}, G_1,\dots, G_{n-1})$ and
\begin{equation}
\label{def:HH_0-HH_1}
\begin{aligned}
\HH_0(\hat r, x_n,\hat \theta,\theta_n,y,G) & = \frac{1}{2\mu_n}\left(y_n^2+ \frac{x^4G_n^2}{4}\right) - \U_0 (\hat r,x_n, \hat \theta, \theta_n), \\
\HH_1 (\hat r, \hat \theta, \hat y, \hat G) & = \sum_{j=1}^{n-1} \frac{1}{2\mu_j}\left(y_j^2+ \frac{G_j^2}{ r_j^2}\right)
-  \U_1 (\hat r, \hat \theta).
\end{aligned}
\end{equation}

Once this notation has been introduced, the equations of motion are:
\[
\begin{aligned}
\dot r_j & =   \frac{\partial \HH}{\partial y_j} =   \frac{\partial \HH_1}{\partial y_j},
&\dot y_j & = -  \frac{\partial \HH}{\partial r_j} = -  \frac{\partial \HH_1}{\partial r_j}+ \O(x_n^6),\\
\dot \theta_j & =  \frac{\partial \HH}{\partial G_j} =  \frac{\partial \HH_1}{\partial G_j},
 &\dot G_j &= -  \frac{\partial \HH}{\partial \theta_j} = -  \frac{\partial \HH_1}{\partial \theta_j} + \O(x_n^6), \\
\dot x_n & = -\frac{x_n^3}{4}  \frac{\partial \HH}{\partial y_n} = -\frac{1}{4\mu_n} x_n^3 y_n,
 &\dot y_n & = -\frac{x_n^3}{4} \left(- \frac{\partial \HH}{\partial x_n} \right) = -\frac{m_n M_n}{4} x_n^4+ \O(x_n^6),\\
\dot \theta_n & =  \frac{\partial \HH}{\partial G_n} = \frac{1}{4\mu_n} x_n^4 G_n,
 &\dot G_n &= -  \frac{\partial \HH}{\partial \theta_n} = \frac{\partial \U_0}{\partial \theta_n} (\hat r,x_n,\hat \theta,\theta_n) = \O(x_n^6),
\end{aligned}
\]
where $1\le j \le n-1$.

It is clear from the above equations that, for all $(\theta_n^0,G_n^0) \in \TT\times \R$, the set $\Lambda_{\theta_n^0,G_n^0} = \{x_n=0,y_n=0,\theta_n = \theta_n^0, G_n = G_n^0\}$ is invariant. The restriction of the dynamics of the system to $\Lambda_{\theta_n^0,G_n^0}$ is given by the Hamiltonian $\HH_1$ in~\eqref{def:HH_0-HH_1}, of $2(n-1)$ degrees of freedom.

\begin{remark}
\label{Diophantine_tor_in_the_n_body_problem}Notice that Hamiltonian~$\HH_1$, in view of~\eqref{HH-HH_0-HH_1}, is precisely a $n$-body problem in Jacobi coordinates. As a consequence, the flow on~$\Lambda_{\theta_n^0,G_n^0}$ is not complete, if $n\ge 4$, due to the existence of non-collision singularities. However, by Arnold's theorem~\cite{Arnold63b}\footnote{Although Arnold's proof is not valid in the spatial case, due to the resonance discovered by Herman~\cite{Fejoz04}, here we deal with the planar case. Another proof of Arnold's theorem can be found in~\cite{ChierchiaP11}.}, at least for an open set of the masses --- those corresponding to the planetary configuration, that is, with one mass much larger than the rest ---, there are initial conditions in~$\Lambda_{\theta_n^0,G_n^0}$ corresponding to quasiperiodic motions. More concretely, assuming the conditions on the masses required by Arnold's theorem, Hamiltonian~$\HH_1$ has Lagrangian (with respect to the form $\sum_{j=1}^{n-1} (dr_j\wedge dy_j+d\theta_j\wedge dG_j)$)
analytic invariant tori (which, consequently, have dimension $2(n-1)$) with flow conjugated to a rigid rotation with Diophantine frequency vector.
F\'ejoz~\cite{Fejoz14} announced that the same claim holds for any values of the masses, giving rise
to the existence of KAM tori in regions of the phase space corresponding to motions close to ellipses of increasingly large semi-axis.
\end{remark}

Next theorem applies to any analytic invariant maximal tori of $\HH_1$ carrying a Diophantine
rotation. Arnold's theorem ensures that the set of such tori is non empty. But $\HH_1$ may have other Diophantine invariant tori. For instance, those  around
normally elliptic periodic orbits of~$\HH_1$.

\begin{theorem}
\label{thm:tori_at_infinity}
Let $\T$ be any analytic invariant $2(n-1)$-dimensional tori of $\HH_1$ with Diophantine frequency vector $\omega^0 \in \R^{2(n-1)}$.
Then, for any $(\theta_n^0,G_n^0) \in \TT\times \R$, the set
\[
\widetilde \T_{\theta_0,G_0} = \{(\hat r, x_n,\hat \theta,\theta_n,y,G) \mid x_n = y_n = 0,\; \theta_n = \theta_n^0, \;
G_n = G_n^0, (\hat r,\hat \theta, \hat y, \hat G) \in \T\}
\]
is a parabolic $2(n-1)$-dimensional invariant tori of $\HH$ with dynamics conjugated to a rigid rotation
with frequency vector $\omega^0$ and with parabolic stable and unstable manifolds, $W^{u,s}_{\theta_n^0,G_n^0}$, which depend analytically on $(\theta_n^0,G_n^0)$. The stable manifold admits a parametrization of the form
\begin{equation}
\label{eq:parametrization_stable_manifold_original_variables}
K_{\theta_n^0, G_n^0}:
\begin{pmatrix}
u \\ \varphi
\end{pmatrix}
\in [0,u_0) \times \TT^{2n-1} \mapsto
\begin{pmatrix}
\hat r (u,\phi,\theta_n^0, G_n^0) \\
(m_n M_n)^{1/4} u + \widetilde \O (u^2) + \O(u^3) \\
\hat \theta (u,\phi,\theta_n^0, G_n^0) \\
\theta_n^0 + \widetilde \O (u^3) + \O(u^4) \\
\hat y (u,\phi,\theta_n^0, G_n^0) \\
(m_n M_n)^{-1/2} \mu_n u + \widetilde \O (u^2) + \O(u^3) \\
\hat G (u,\phi,\theta_n^0, G_n^0) \\
G_n^0 + \widetilde \O (u^3) + \O(u^4) \\
\end{pmatrix},
\end{equation}
where $\O^* (u^k)$ denotes a function of order $u^k$ independent of $\theta_n^0$, $G_n^0$ and $\varphi$,
such that
\[
\Phi_t ( K_{\theta_n^0, G_n^0}(u,\varphi)) = K_{\theta_n^0, G_n^0}(\tilde \Phi_t (u;\theta_n^0,G_n^0),\varphi + \omega t), \qquad t \ge 0,
\]
where $\Phi_t$ is the flow of Hamiltonian $\HH$ and $\tilde \Phi_t$ is the flow of
\[
\dot u = - \frac{1}{4} u^4 + b(\theta_n^0, G_n^0) u^7,
\]
for some analytic function $b(\theta_n^0, G_n^0)$.

Furthermore, the set
\[
\widehat \T_{G_n^0} =  \bigcup_{\theta_n^0 \in \TT} \widetilde \T_{\theta_n^0,G_n^0}
\]
is an parabolic $(2n-1)$-dimensional invariant tori of $\HH$.
It has parabolic Lagrangian invariant stable and unstable manifolds, $W^{u,s}_{G_n^0} = \bigcup_{\theta_n^0 \in \TT} W^{u,s}_{\theta_n^0,G_n^0}$. The stable manifold has a parameterization $\widetilde K_{G_n^0}(u,\theta_n^0,\phi) = K_{\theta_n^0,G_n^0}(u,\phi)$ satisfying
\[
\Phi_t ( \widetilde K_{ G_n^0}(u,\theta_n^0,\varphi)) = \widetilde K_{ G_n^0}(\tilde \Phi_t (u;\theta_n^0,G_n^0),\theta_n^0,\varphi + \omega^0 t), \qquad t \ge 0.
\]

The analogous claim holds for the unstable manifold.

\end{theorem}

\begin{remark}
\label{remark:no_diffusion}
From Theorem~\ref{thm:tori_at_infinity}, we obtain one parameter families of tori, $G_0\mapsto \widehat \T_{G_0}$, which depend analytically on $G_n^0$, with stable and unstable Lagrangian invariant manifolds. It should be noted that in these families $W^s \widehat \T_{G_n^0}$ does not intersect $W^u \widehat \T_{{G_n^0}'}$, if $G_n^0 \neq {G_n^0}'$. Indeed, Hamiltonian $\HH$ has an additional conserved quantity, the total angular momentum, given by $\GGG= \sum_{j=1}^n G_j$. But $\hat \GGG = \sum_{j=1}^{n-1} G_j$ is a conserved quantity of $\HH_1$, which, since $\dot G_{n \mid x_n=y_n=0} = 0$, implies that
\[
\GGG_{\mid \widehat \T_{G_n^0}} = \hat \GGG_{\mid \T} + G_n^0
\]
and the same happens on the stable and unstable manifolds of  $\widehat \T_{G_n^0}$. Hence, the invariant manifolds of different tori in a family lie on different level sets of the total angular momentum.
\end{remark}

\begin{proof}[Proof of Theorem~\ref{thm:tori_at_infinity}]
Since $\T$ is analytic, invariant and its dynamics is conjugated to a rigid rotation of frequency vector
$\omega^0$, it is Lagrangian. Then, by Weinstein's theorem, there exist analytic symplectic action angle coordinates $(\phi,\rho)\in \TT^{2(n-1)}\times \R^{2(n-1)}$ in which
$\T = \{\rho =  0\}$, or, equivalently, $\HH_1$ in these variables becomes
\[
\HH_1(\phi,\rho) = \langle \omega^0 ,\rho \rangle + \O(\rho^2).
\]
Since $\omega^0$ is Diophantine, we can perform five steps of averaging, if necessary, to assume that
\[
\HH_1(\phi,\rho) = \langle \omega^0 ,\rho \rangle + \sum_{\ell=2}^6 c_{\ell} \rho^\ell+\O(\rho^7),
\]
where $c_{\ell}$ are constant $\ell$-linear forms.
The change of variables
\[(\phi,x_n,\theta_n,\rho,y_n,G_n) \mapsto
(\hat r(\phi,\rho), x_n,\hat \theta(\phi,\rho),\theta_n,\tilde y(\phi,\rho),y_n, \tilde G(\phi,\rho), G_n)
\]
is symplectic (preserves the form~\eqref{form_McGehee}).
We will denote by $\widetilde \HH$ the Hamiltonian in the new variables. Let $\widetilde  \HH_0$ and $\widetilde \HH_1$ be
\[
\begin{aligned}
\widetilde  \HH_0(\phi,x_n,\theta_n,\rho,y_n,G_n) & = \HH_0(\hat r(\phi,\rho), x_n,\hat \theta(\phi,\rho),\theta_n,
\tilde y(\phi,\rho),y_n, \tilde G(\phi,\rho), G_n),\\
\widetilde  \HH_1(\phi,\rho) & = \HH_1(\hat r(\phi,\rho),\hat \theta(\phi,\rho),\hat y(\phi,\rho),\hat G(\phi,\rho))
= \langle \omega^0 ,\rho \rangle + \O(\rho^2).\\
\end{aligned}
\]
We have that $\widetilde  \HH =\widetilde  \HH_0+\widetilde \HH_1$.

Since $\omega^0$ is Diophantine, by Remark~\ref{averaging} below, we can  assume that, in a new set of canonical variables,
\[
\frac{\partial \widetilde \HH}{\partial \phi} = \O(\rho^{12})+ \O(x_n^{12}).
\]

\begin{remark}
\label{averaging}
The averaging procedure can be performed using generating functions in the following way. Given a function
\begin{equation*}
\S(\Phi,\rho,y_n,X_n,\Theta_n,G_n) = \Phi \rho + \Theta_n G_n+ \frac{2}{X_n^2}y_n +  S(\Phi,\rho,y_n,X_n),
\end{equation*}
if the equations
\begin{equation}
\label{eq:generating_function_symplectic_change}
\begin{aligned}
\phi & = \Phi + \frac{\partial S}{\partial \rho} (\Phi,\rho,y_n,X_n) ,
&R & = \rho + \frac{\partial S}{\partial \Phi} (\Phi,\rho,y_n,X_n), \\
\frac{2}{x_n^2} & = \frac{2}{X_n^2} + \frac{\partial S}{\partial y_n} (\Phi,\rho,y_n,X_n),
&\frac{4}{X_n^3} Y_n & = \frac{4}{X_n^3} y_n - \frac{\partial S}{\partial X_n} (\Phi,\rho,y_n,X_n), \\
\Theta_n & = \theta_n,
&\tilde G_n & = G_n,
\end{aligned}
\end{equation}
define a close to the identity map $T:(\phi,\rho,x_n,y_n, \theta_n,G_n) \mapsto (\Phi,R,X_n,Y_n,\Theta_n, \tilde G_n)$, then $T$ preserves the $2$-form
\begin{equation}
\label{two_form_McGehee_action_angle}
\omega =  \sum_{j=1}^{2(n-1)} d\phi_j \wedge d\rho_j -\frac{4}{x_n^3}\, dx_n \wedge dy_n + d\theta_n\wedge dG_n.
\end{equation}
Indeed, $T$ preserves $\omega$ if and only if $\omega - T^* \omega = 0$. Since $\omega - T^* \omega = d \sigma$, where
\[
\sigma = \phi\, d\rho + R\, d\Phi +\frac{2}{x_n^2} \,dy_n - \frac{4}{X_n^3} Y_n \,dX_n+\theta_n\, dG_n + \tilde G_n \, d\Theta_n,
\]
one has that $\sigma = d \S$.

Now, assume that the Hamiltonian $\widetilde \HH$ has a monomial of the form $a(\phi)x_n^i y_n^j \rho^k$,
where $k = (k_1, \dots, k_{2(n-1)})$. Taking $S$ as
\[
S (\Phi,r,y_n,X_n) = A(\Phi) X_n^i y_n^j \rho^k,
\]
equations~\eqref{eq:generating_function_symplectic_change} do define a close to the identity map.
Indeed, equations~\eqref{eq:generating_function_symplectic_change} become
\begin{equation*}
\begin{aligned}
\phi & = \Phi + kA(\Phi) X_n^i y_n^j \rho^{k-1},
&R & = \rho + \nabla A(\Phi) X_n^i y_n^j \rho^k, \\
\frac{2}{x_n^2} & = \frac{2}{X_n^2} + j A(\Phi) X_n^i y_n^{j-1} \rho^k ,
&\frac{4}{X_n^3} Y_n & = \frac{4}{X_n^3} y_n - i A(\Phi) X_n^{i-1} y_n^j \rho^k, \\
\Theta_n & = \theta_n,
&\tilde G_n & = G_n.
\end{aligned}
\end{equation*}
They define a close to the identity map near $x_n=y_n=0$, $\rho_n=0$. Hence,
\begin{equation*}
\begin{aligned}
\phi & = \Phi + kA(\Phi) X_n^i Y_n^j R^{k-1} + \O_{i+j+|k|}, \\
x_n & = X_n - \frac{j}{4} A(\Phi) X_n^{i+3} Y_n^{j-1} R^k + \O_{i+j+|k|+3}, \\
\theta_n & = \Theta_n, \\
\rho & = R - \nabla A(\Phi) X_n^i Y_n^j R^k + \O_{i+j+|k|+1},\\
y_n & =  Y_n + \frac{i}{4} A(\Phi) X_n^{i+2} Y_n^j R^k + \O_{i+j+|k|+3},\\
G_n & = \tilde G_n,
\end{aligned}
\end{equation*}
where $\O_{i+j+|k|} = \O(\|(R,X_n,Y_n)\|^{i+j+|k|})$ is symplectic with respecto to $\omega$. Applying this transformation to $\widetilde \HH$, the coefficient of the monomial $X_n^i Y_n^j R^k$ is
\[
\omega^0 \nabla A(\Phi) + a(\Phi).
\]
Since $\omega^0$ is Diophantine, we can choose $A$ such that this monomial does not depend on $\Phi$. Since the dependence on $\phi$ starts at order  at least $3$, one can proceed recursively.
\end{remark}

The equations of motion of $\widetilde \HH$ are
\begin{equation}
\label{eqofmotionalmostdef}
\begin{aligned}
\dot \phi & =   \frac{\partial \widetilde \HH}{\partial \rho} =   \omega^0+\O(\rho) + \O(x_n^6),
&\dot \rho & = -  \frac{\partial \widetilde \HH}{\partial \phi} =  \O(\rho^{12})+ \O(x_n^{12}),\\
\dot x_n & = -\frac{x_n^3}{4}  \frac{\partial \widetilde \HH}{\partial y_n} = -\frac{1}{4\mu_n} x_n^3 y_n,
 &\dot y_n & = \frac{x_n^3}{4} \frac{\partial \widetilde \HH}{\partial x_n} = -\frac{m_n M_n}{4} x_n^4+ \O(x_n^6),\\
\dot \theta_n & =  \frac{\partial \widetilde \HH}{\partial G_n} = \frac{1}{4\mu_n} x_n^4 G_n,
 &\dot G_n &= -  \frac{\partial \widetilde \HH}{\partial \theta_n} =  \O(x_n^6).
\end{aligned}
\end{equation}
In the following, we will perform some changes of variables to the system~\eqref{eqofmotionalmostdef} in order to transform it into a system satisfying the hypotheses of Theorem~\ref{main_theorem_flow_case}. In this way we will obtain the stable manifold of the torus. In order to obtain the unstable manifold, first we change the sign of time and then apply the analogous changes of variables.
We start by  rescaling the variables $x_n$, $y_n$ and $G_n$ by defining
\[
\tilde x = (m_n M_m)^{-1/4}x_n, \qquad \tilde y = (m_n M_m)^{1/2} \mu_n^{-1}y_n, \qquad  \tilde G = \mu_n^{-1} G_n.
\]
Then, we introduce $\alpha = \theta_n + \tilde G \tilde y$ and  we define
\[
q = \frac{1}{2} (\tilde x + \tilde y), \qquad p = \frac{1}{2} (\tilde x-\tilde y).
\]
Then, denoting $z = ( \alpha, \tilde G)$, equations~\eqref{eqofmotionalmostdef} become
\[
\begin{aligned}
\dot q & =  -\frac{1}{4}(q+p)^3 \left(q+ \O((q+p)^3)\right),
 &\dot p & = \frac{1}{4}(q+p)^3 \left(p+ \O((q+p)^3)\right),\\
 \dot z & =  \O((q+p)^6,\rho^6),
 &\dot \rho & =  \O((q+p)^{12},\rho^{12}), \\
 \dot \phi & =     \omega^0 + \O((q+p)^6,\tilde \rho).
\end{aligned}
\]
Finally, we choose $\alpha^0$ and $\tilde G^0$ (or equivalently, $\theta_n^0$ and $G_n^0$, and, then, $\alpha^0 = \theta_n^0$, $\tilde G^0 = \mu_n^{-1} G_n^0$),
define $z^0 = (\alpha^0,\tilde G^0)$ and introduce for $q+p>0$ (equivalently, for $x_n>0$)
\[
\tilde z = \frac{1}{q+p} (z-z^0), \qquad \tilde \rho = \frac{\rho}{6(q+p)^6} .
\]
After this last change, denoting $w=(\tilde z, \tilde \rho)$, equations~\eqref{eqofmotionalmostdef} become
\begin{equation}
\label{eqofmotiondef}
\begin{aligned}
\dot q & =  -\frac{1}{4}(q+p)^3 \left(q+ \O((q+p)^3)\right),
 &\dot p & = \frac{1}{4}(q+p)^3 \left(p+ \O((q+p)^3)\right),\\
 \dot w & =  \frac{1}{4}(q+p)^2(p-q)\tilde w + \O((q+p)^5),
 &\dot \phi & =     \omega^0 + \O((q+p)^6).
\end{aligned}
\end{equation}
This system satisfies the hypotheses of Theorem~\ref{main_theorem_flow_case} with $\lambda = (\alpha^0,\tilde G^0)$,  $N=4$,
$a(\phi,\lambda) = 1/4$ and $P=6$. Hence, the invariant torus $\{q=p=0, \; \tilde r = 0\}$
has parabolic stable invariant manifolds parametrized by some embedding $K^s(u,\varphi,\lambda)$, analytic with respect to  $(u,\varphi,\lambda)$ in some complex domain containing $(0,\delta_0)\times \TT \times \{(\alpha^0,\tilde G^0)\}$,
$C^{\infty}$ at $\{u=0\}$, with $K^s(0,\tilde \phi,\lambda) = (0,0,0,\tilde \phi)$, $\partial_x K^s (0,\phi,\lambda) = (0,1,0,0)^{\top}$. Moreover, taking into account that the dependence of the $(q,p)$ components of the vector field defined by~\eqref{eqofmotiondef} on $(w,\phi,\lambda)$ starts at order $6$, while $N=4$, we have that the parametrization of the stable manifold has the form
\[
(q,p,w,\phi) = K^s(u,\varphi,\lambda)
=
\begin{pmatrix}
\O^* (u^2) + \O(u^3) \\
u + \O^* (u^2) + \O(u^3) \\
\O (u^2) \\
\varphi + \O (u)  \\
\end{pmatrix}, \qquad (u,\phi) \in [0,u_0)\times \TT^{2n-1},
\]
where $\O^* (u^2)$ denotes a function of order $u^2$ independent of $\varphi$ and $\lambda$. Going back to the variables
$(\phi,x_n,\theta_n,\rho,y_n,G_n)$ in which~\eqref{eqofmotionalmostdef} is written, we have that
\begin{equation}
\label{param_stable_manifold_infinity}
K^s (u,\varphi,\theta_n^0,G_n^0)
=
\begin{pmatrix}
\varphi+ \O(u) \\
(m_n M_n)^{1/4} u + \O^*(u^2) + \O(u^3)\\
\theta_n^0 +  \O (u^3)  \\
\O (u^8) \\
(m_n M_n)^{-1/2} \mu_n u + \O^* (u^2) + \O(u^3) \\
G_n^0+ \O (u^3)
\end{pmatrix},
\end{equation}
where $(\theta_n^0, G_n^0)$ are parameters. The embedding $K^s$ satisfies the invariance equation
\begin{equation}
\label{inv_eq_Lagrangian}
\Psi_t \circ K^s (u,\varphi,\theta_n^0,G_n^0) =
K^s (\widetilde \Psi_t(u,\theta_n^0,G_n^0) ,\varphi+\omega^0 t,\theta_n^0,G_n^0),
\end{equation}
where $\Psi_t$ is the flow of~\eqref{eqofmotionalmostdef} and $\widetilde \Psi_t$ is the flow of the equation
\[
\dot u = - \frac{1}{4} u^4 + b(\theta_n^0, G_n^0) u^7,
\]
obtained by applying Theorem~\ref{formal_part_theorem_flows} to~\eqref{eqofmotiondef}.
Going back to the original variables, we obtain expression~\eqref{eq:parametrization_stable_manifold_original_variables}.

It only remains to check that, for each $G_n^0$, the parametrization
\[
K:(u,\theta_n^0,\varphi) \mapsto K^s (u,\varphi,\theta_n^0,G_n^0)
\]
of $\widehat \T_{G_n^0}$ defines a Lagrangian manifold, that is, that the $2$-form $\omega$ in~\eqref{two_form_McGehee_action_angle} vanishes identically on $\widehat \T_{G_n^0}$. We will check that
\[
\omega ( \partial_u K,\partial_{\theta_n^0} K) = \omega ( \partial_u K,\partial_{\varphi_i} K) =
\omega ( \partial_{\theta_n^0} K,\partial_{\varphi_i} K) = \omega ( \partial_{\varphi_j} K,\partial_{\varphi_i} K) = 0,
\]
where $1\le i, j \le 2(n-1)$. We check the equality for $\omega ( \partial_u K,\partial_{\theta_n^0} K)$, being the argument for the rest identical.

First we remark that, since $G_n^0$ is fixed and $\theta_n^0  \in \TT$, for any $a_- < 1/4 < a_+$ and any $0 < \alpha < 1$, there exists $u_0 >0$ such that for all $u \in [0,u_0)$ and $t\ge 0$,
\begin{equation}
\label{dynamics_on_the_manifold}
\begin{aligned}
\frac{u}{(1+3a_+ u^3 t)^{1/3}}  & \leq  \widetilde \Psi_t (u,\theta_n^0,G_n^0) \leq
\frac{u}{(1+3a_- u^3 t)^{1/3}}, \\
\frac{1}{(1+3a_+ u^3 t)^{1/(3\alpha)}} & \leq  \partial_u \widetilde \Psi_t (u,\theta_n^0,G_n^0) \leq
\frac{1}{(1+3a_- u^3 t)^{\alpha/3}} .
\end{aligned}
\end{equation}
Since $\Psi^*_t \omega = \omega$, taking derivatives at~\eqref{inv_eq_Lagrangian} and~\eqref{param_stable_manifold_infinity}, we have that, for all $t\ge 0$,
\begin{multline*}
|\omega ( \partial_u K(u,\theta_n^0,\varphi),\partial_{\theta_n^0} K(u,\theta_n^0,\varphi)) | \\
\begin{aligned}
= & | \omega ( \partial_u K(\widetilde \Psi_t (u),\theta_n^0,\varphi+\omega^0 t)\partial_u \widetilde \Psi_t (u,\theta_n^0,G_n^0),\partial_{\theta_n^0} K(\widetilde \Psi_t (u),\theta_n^0,\varphi+\omega^0 t))| \\
 \le &  C \left( \left| \frac{\partial_u \widetilde \Psi_t (u,\theta_n^0,G_n^0)}{\widetilde \Psi_t (u,\theta_n^0,G_n^0)} \right|\,| \widetilde \Psi_t (u,\theta_n^0,G_n^0)| +  | \widetilde \Psi_t (u,\theta_n^0,G_n^0)|^2 +  | \widetilde \Psi_t (u,\theta_n^0,G_n^0)|^8
 \right)
\end{aligned}
\end{multline*}
Hence, by~\eqref{dynamics_on_the_manifold}, we have that
\[
\omega ( \partial_u K(u,\theta_n^0,\varphi),\partial_{\theta_n^0} K(u,\theta_n^0,\varphi))
 = \lim_{t\to \infty} \Psi^*_t \omega ( \partial_u K(u,\theta_n^0,\varphi),\partial_{\theta_n^0} K(u,\theta_n^0,\varphi)) =0.
\]
\end{proof}

%% file: BFM_Tori_Proofs_Reduit-Diophantine-v3.tex
\section{Proofs of the results. Map case}\label{sec:proofsmaps}

Here we prove the results stated in Section~\ref{sec:mainresults}. We first need to introduce some technical notation and preliminary considerations.
This is done in Section~\ref{sec:notproofs} below.
With respect to the proofs of results, in Section~\ref{sec:proofsmapstrue} we prove the existence and regularity results of invariant parabolic
manifolds associated to
normally parabolic tori for analytic maps, Theorem~\ref{thetruemanifold}.
Then, in Section~\ref{sec:proofsmapsformal}, we deal with obtaining formal (or approximated) manifolds, Theorem~\ref{formalmanifold}.
Finally, in Section~\ref{sec:proofcorollary} we prove Corollary~\ref{mapcorollary}.

\subsection{Notation and the small divisors equation}\label{sec:notproofs}
In the proofs of the main results, when doing steps of averaging and when solving cohomological equations we will encounter the so-called small divisors
equation. In the setting of maps the equation we find is
\begin{equation*}
\varphi(\theta + \omega) - \varphi(\theta) = h(\theta),
\end{equation*}
with $h:\TT^{d}\to \R^k$ and $\omega \in \R^d$. When $k=1$ this is a scalar equation but we can also consider vector or matrix equations choosing
$\varphi$ accordingly.

We will find this equation depending on parameters. We are mainly interested in the analytic case, but this equation can also be considered for differentiable
functions. To be concrete we consider $h:\TT^{d}_\sigma\times \Lambda_\C \to \C^k$ and we want to find a solution $\varphi(\th,\lambda)$ of
 \begin{equation} \label{smalldiveq}
\varphi(\theta + \omega,\lambda) - \varphi(\theta,\lambda) = h(\theta,\lambda),
\end{equation}
in a suitable domain. We develop $h$ in Fourier series
$$
h(\theta,\lambda) = \sum_{k\in \Z^d} h_k(\lambda) \text{e}^{2\pi i  k \cdot  \theta}, \qquad k\cdot \theta = k_1 \th_1 + \cdots + k_d\theta_d.
$$
If $h$ has zero average and $ k\cdot  \omega \notin \Z$ for all $k\neq 0$, equation~\eqref{smalldiveq} has a formal solution
\begin{equation*}
\varphi(\theta,\lambda) = \sum_{k\in \Z^d } \varphi_k(\lambda) \text{e}^{2 \pi i k \cdot \theta},\qquad
\varphi_k(\lambda) = \frac{h_k(\lambda)}{1- \text{e}^{2\pi ik\cdot \omega}},
\qquad k\neq 0.
\end{equation*}
All coefficients $\varphi_k$ are uniquely determined except $\varphi_0$ which is free.

We quote the well known result
\begin{theorem}[Small divisors lemma]\label{thsmalldiv}
Let $h:\TT^d_\sigma \times \Lambda_\C \to \R^k$ be analytic with zero average and $\omega$ Diophantine with $\tau \geq d-1$
(see the notation in Section~\ref{subsec:notation}).

Then there exists a unique analytic solution $\varphi:\TT^d_\sigma \times \Lambda_\C \to \R^k$ of~\eqref{smalldiveq} with zero average and
$$
\sup_{(\th,\lambda) \in \TT^d_{\sigma-\delta} \times \Lambda_\C } \|\varphi(\th,\lambda)\|
\leq  C \delta^{-\tau} \sup_{(\th,\lambda) \in \TT^d_{\sigma} \times \Lambda_\C } \|h(\th,\lambda)\|,\qquad 0<\delta < \sigma,
$$
where $C$ depends on $\tau$ and $d$ but not in $\delta$.
\end{theorem}
Two analytic soluctions of~\eqref{smalldiveq} differ by a function of $\lambda$. The proof with close to optimal estimates
is due to Russmann~\cite{Russmann75}. See also de la Llave~\cite{Llave01} and Figueras et al~\cite{FiguerasHL18} for a proof with explicit and very sharp estimates for applications
in Computer Assisted Proofs. For the proof in presence of parameters one only has to take into account that
$$
h_k(\lambda) =  \int_{\TT^d} h(\th,\lambda) e^{- 2\pi i k\cdot \th}\, d\th
$$
and proceed as in the usual proof.

We will denote by $\SD(h)$ the unique solution of equation~\eqref{smalldiveq} with zero average.

To finish this introductory section, we set the Banach spaces we will work with.
Given $k\in \N$, $\beta,\rho,\sigma>0$ and $\Lambda_{\C}$ a complex extension of $\Lambda$, we introduce for $q\in \R$,
$$
\X_{q}=\left \{ \Delta : S(\beta, \rho) \times \TT^d_\sigma \times \Lambda_\C \to \C^{k} \mid \, \text{analytic, }\sup_{(x,\th)\in S\times \TT^d_\sigma} \frac{|\Delta (x,\th)|}{|x|^q} <\infty\right\}
$$
endowed with the norm
$$
\| \Delta \| = \sup_{(x,\th,\lambda)\in S\times \TT^d_\sigma\times \Lambda_\C} \frac{|\Delta (x,\th)|}{|x|^q}.
$$

We recall that, as we pointed out in Section~\ref{subsec:notation}, we omit the parameters $\beta,\rho$ in $S$.
In addition, from now we will omit the dependence on $\lambda$ of our notation.

\subsection{Existence of a stable manifold. Proof of Theorem~\ref{thetruemanifold}}\label{sec:proofsmapstrue}
In this section we assume that $F$ is analytic in a neighbourhood of the origin having the form~\eqref{system} with $P=N$. The case
$P>N$ is also included since $h_N\equiv 0$ fits in our setting. We will prove
that, given an approximated parameterization of an invariant manifold up to some order $Q\geq N$, there is a parameterization of a true
invariant manifold whose expansion coincides with that of the approximation until order $(\O(|x|^Q), \O(|x|^{Q}), \O(|x|^{Q-1}))$.

More concretely, we assume that there exists $K^{\leq}=(K^{\leq}_x,K^{\leq}_y, K^{\leq}_\th)$ and $R^{\leq}=(R^{\leq}_x,R^{\leq}_\th)$ such
that
\begin{equation}\label{defEleqsec}
E^{\leq}:=F\circ K^{\leq} - K^{\leq} \circ R^{\leq}
\end{equation}
satisfies
\begin{equation*}
E^{\leq}=(E^{\leq}_{x},E^{\leq}_{y},E^{\leq}_\th) = (\O(|x|^{Q+N}), \O(|x|^{Q+N}), \O(|x|^{Q+N-1})).
\end{equation*}
We assume that the domain of $K^{\leq}$ and $R^{\leq}$ is $S(\beta_0,\rho_0)\times \TT_{\sigma'}^d \times \Lambda_\C$
for some $\beta_0,\rho_0,\sigma'>0$.

According to the parameterization method, to obtain the invariant manifold and the other conclusions of Theorem~\ref{thetruemanifold},
we look for $\Delta=(\Delta_x,\Delta_y,\Delta_\th) \in \X_{Q+1}\times \X_{Q+1}\times \X_Q$ such that, for some $\beta,\rho>0$
and $\Lambda_\C' \subset \Lambda_\C$ a complex extension of $\Lambda$ (to be determined along the proof), we have that:
\begin{equation}\label{invcondsection}
F\circ (K^{\leq} + \Delta) = (K^{\leq} + \Delta)\circ R^{\leq},\qquad \text{in  }\quad S(\beta,\rho)\times \TT^d_{\sigma'} \times \Lambda_\C'.
\end{equation}
That is, we slightly modify $K^{\leq}$ while maintaining the same reparametrization $R^{\leq}$.
We can not guarantee that the domain of $\Delta$ is the same as the one for $K^{\leq}$,
however we maintain the same width in the complex strip for $\th$.

\subsubsection{Preliminary reductions}
	To determine the existence of $\Delta$, it is convenient to perform some changes
	of variables to $F$ to put it in a more suitable form to deal with the estimates.
	These changes are
	two steps of averaging to kill the dependence on $\th$ of
	the coefficients $a(\th),B(\th)$, one rescaling to make $\overline{a}(\lambda)$ independent of $\lambda$,
	a linear change of the variable $y$ to transforme $B$ to a close to diagonal matrix and a rescaling of the $y$ variables.
	Since the dependence on $\lambda$ is a local property, we will work with some $\Lambda_{\C}'$ that will be a small neighborhood of a
	fixed value $\lambda=\lambda_0$. However we will put no conditions on $\lambda_0$, apart from being real.
\begin{lemma}\label{changevariables}
Let $F$ be a map of the form~\eqref{system} satisfying the conditions (i)-(v) in Section~\ref{subsec:mapcase} having a homomorphic
analytic extension to $\U_\C \times \TT_\sigma^d$, $\lambda_0\in \Lambda$ and $0<\e<1$. Then, there exists a real analytic change of variables
$T(x,y,\th)$, depending on $\e$, $T:\C^{1+m}\times \TT_{\sigma'}^d \times \Lambda_\C'\to \C^{1+m} \times \TT_{\sigma'}^d$ such that
$F$, in the new variables, has the form
\begin{equation}\label{FinalexpF}
	\left ( \begin{array}{c} x\\ y \\ \th\end{array}\right ) \mapsto
	\left ( \begin{array}{c} x-x^N + \f_N(x, y,\th) + \f_{\geq N+1} (x,y,\theta)\\
	y+ x^{N-1} \J y + \g_{N}(x,y, \th) +
	\g_{\geq N+1} (x, y,\theta)\\
	\th + \omega + \h_N(x, y,\th)+\h_{\geq N+1} (x, y,\th)
	\end{array}
	\right )
	\end{equation}
	with
	\begin{enumerate}
	\item $\J=J(\lambda)$ is close to the Jordan form of $\overline{B}(\lambda_0)$ with arbitrary small terms off the
	diagonal.
	\item $\f_N,\g_N, \h_N$ are homogeneous polynomials of order $N$ with $\f_N(x,0,\th)=0$, $\g_N(x,0,\th)=0$, $\Dy\g_N(x,0,\th)=0$,
	and $\f_{\geq N+1}, \h_{\geq N+1} =\O(\|(x,y)\|^{N+1})$.
	\item The term $\f_{N-1,1}(\th) x^{N-1} y$ of $\f_N$ is $\e f_{N-1,1}(\th) x^{N-1} y$.
	\item The terms $\g_{N}(x,y,\th)$ and
	$\g_{\geq N+1}(x,y,\th)$ are of the form
	\begin{equation}\label{gN+1escalat}
	\begin{aligned}
	\g_{N}(x,y,\th) &= \e \|y\|^2 \O(\|(x,y)\|^{N-2}), \\
	\g_{\geq N+1}(x,y,\th) &= \e^{-1} \O(|x|^{N+1}) + \|y\| \O(|x|^{N}) + \e \O(\|(x,y)\|^{N+1}).
	\end{aligned}
	\end{equation}
	\end{enumerate}
\end{lemma}
\begin{proof}
Let $0<\sigma'<\sigma$.
	A change of the form $T_1(x,y,\th)=(x+c_1(\th) x^N,y,\th)$ with $c_1:\TT_{\sigma'}^d\to \C$, applied to $F$ preserves
	the terms of order $N$ of $F_x,F_y$ and the ones of order $P=N$ of $F_\th$
	except the monomial $-a(\theta) x^N$ of $F_x$ which becomes
	$$
	\big [c_1(\th) - c_1(\th + \omega) - a(\th)\big ]x^N.
	$$
	We kill the oscillating part $\widetilde {a}$ of $a$ by applying the small divisors lemma. We choose
	$c_1 = \SD (\widetilde {a})$, hence the corresponding term becomes $-\overline{a} x^N$.
	
	In the same way, the change $T_2(x,y,\th) = (x, y+C_2(\th) x^{N-1} y, \th)$
	transforms the term $x^{N-1} B(\th) y$ of $F_y$ to
	$$
	x^{N-1}\big [C_2(\th) - C_2(\th + \omega) + B(\th)\big ]y
	$$
	while keeping unchanged the other terms of order $N$ (of $F_x,F_y$) and order
	$P=N$ (of $F_\th$). We choose $C_2=\SD(\widetilde {B})$ defined on $\TT^d_{\sigma'}$, so that the mentioned term becomes
	$x^{N-1} \overline{B} y$.
		
	To simplify the proof, we make $\overline{a}$  independent
	of the parameter $\lambda$. For that we scale the $x$-variable by
	$T_3(x,y,\th)= (\mu x,y,\th)$ with $\mu=(\overline{a}(\lambda))^{-\alpha}$ and $\alpha=1/(N-1)$.
	We obtain the new constant $\overline{a}=1$. We emphasize that, when $\lambda\in \Lambda \subset \R^p$,
	$\overline{a}(\lambda)>0$, therefore, for a suitable complex extension $\Lambda_{\C}'$ of $\Lambda$,
	$\Re ( \overline{a}(\lambda))>0$ if $\lambda\in \Lambda_{\C}'$ and the rescaling is well defined.
	
	Next, let $D\in L(\R^{m},\R^m)$ and the change $T_4(x,y,\th)=(x,Dy,\th)$. The
	transformed map is
	$$
	\left ( \begin{array}{c} x\\ y \\ \th\end{array}\right ) \mapsto
	\left ( \begin{array}{c} x-x^N + f_N(x,Dy,\th) + f_{\geq N+1} (x,D y,\theta) \\
	y+ x^{N-1} D^{-1} \overline{B} Dy + D^{-1} g_{N}(x,Dy, \th) + D^{-1} g_{\geq N+1} (x,D y,\theta) \\
	\th + \omega + h_N(x,Dy,\th)+h_{\geq N+1} (x,D y,\theta)
	\end{array}
	\right ).
	$$
	We choose $D$ as the linear change that transforms $\overline{B}(\lambda_0)$ to its Jordan form, $J(\lambda_0)$, with arbitrarily small terms off
	the diagonal. Therefore, taking $\Lambda_{\C}'$ small, $J(\lambda)=D^{-1} \overline{B}(\lambda) D$ will be close to $J(\lambda_0)$.
	
	Finally we make the change $T_5(x,y,\th) = (x,\e y, \th)$. The transformed map
	is
	$$
	\left ( \begin{array}{c} x\\ y \\ \th\end{array}\right ) \mapsto
	\left ( \begin{array}{c} x-x^N + f_N(x,\e Dy,\th) + f_{\geq N+1} (x,\e Dy,\theta)\\
	y+ x^{N-1} \J y + \e^{-1} D^{-1} g_{N}(x,\e Dy, \th) +
	\e^{-1} D^{-1} g_{\geq N+1} (x,\e Dy,\theta)\\
	\th + \omega + h_P(x,\e Dy,\th)+h_{\geq P+1} (x, \e Dy,\th)
	\end{array}
	\right ).
$$
	To finish, recalling that $f_{N}(x,0,\th)=0$, $g_{N}(x,0,\th)=0$ and $\Dy g_{N}(x,0,\th)=0$, we obtain
	the conclusions for $\f_N,\g_N$. The expression for $\g_{\geq N+1}$ follows immediately.

	The claimed change of variable is the composition $T=T_5 \circ T_4 \circ T_3 \circ T_2 \circ T_1$.
	\end{proof}
\begin{remark}\label{remarkgN+1}
The first two terms of $\g_{\geq N+1}$ in~\eqref{gN+1escalat} will be controlled by working in a
	small sector such that $|x|<\rho$ and $\e^{-1} \rho^{N+1} $ is small.
\end{remark}
	Let us denote by $F_1$ the transformed map:
	$F_1 = T^{-1} \circ F \circ T$.
	Assume that $K^{\leq}$ and $R^{\leq}$ satisfy the conditions of
	Theorem~\ref{thetruemanifold}. From
	$$
	T^{-1}\circ F \circ T \circ T^{-1}\circ K^{\leq} = T^{-1} \circ (K^{\leq}
	\circ R^{\leq} + E^{\leq} ),
	$$
	we write
	$$
	F_1 \circ K_1^{\leq} = K_1^{\leq} \circ R^{\leq} + E_1^{\leq},
	$$
	where
	$$
	K_1^{\leq} = T^{-1}\circ K^{\leq},\qquad
	E_1^{\leq} = T^{-1}\circ (K^{\leq} \circ R^{\leq} + E^{\leq} ) - T^{-1} \circ
	K^{\leq} \circ R^{\leq}.
	$$
	Since $E_1^{\leq} =DT^{-1} (K^{\leq} \circ R^{\leq}) E^{\leq} +
	\mathcal{O}(\|E^{\leq}\|^2)$ we have that the components of $E^{\leq}_1$
	have the same order as the ones of $E^{\leq}$.
	However, the first component of $K_1^{\leq}$ is $\mu^{-1} x+\O(|x|^2)$ instead of $x+O(|x|^2)$.
	For that reason we define $K_2^{\leq}(x,\th) = K_1^\leq (\mu x ,\th)$ and
	\begin{align*}
	R_2^\leq (x,\th) & =\mu^{-1}R^\leq (\mu x,\th) \\ & = x - \overline{a}(\lambda) \mu^{N-1} x^{N} + \O(|x|^{N+1}) = x-x^N + \O(|x|^{N+1})
	\end{align*}
and we observe that
$$
F_1 \circ K_2^{\leq} (x,\th) - K_2^\leq \circ R_2^\leq (x,\th) = F_1 \circ K_1^{\leq} (\mu x,\th) - K_1^{\leq} \circ R^\leq (\mu x,\th) = E_1^\leq(\mu x,\th)
$$
which again has the same orders as the ones of $E^{\leq}$.
	
	We notice that, if $F,K^{\leq},R^{\leq}$ are under the conditions of
	Theorem~\ref{thetruemanifold}, the same happens for
	$F_1$, $K_2^{\leq}$ and $R_2^{\leq}$. Then if we can find
	$\Delta_2\in \X_{Q+1} \times\X_{Q+1} \times \X_Q$ such that
	$$
	F_1 \circ(K_2^{\leq} + \Delta_2) = (K_2^{\leq} + \Delta_2) \circ R_2^{\leq},
	$$
	defining $\Delta_1(x,\th) = \Delta_2(\mu^{-1} x,\th)$, the condition
	$$
	F \circ T \circ (T^{-1} \circ K^{\leq} + \Delta_1 )(\mu x, \th) = T \circ (T^{-1} \circ
	K^{\leq} + \Delta_1)\circ R^{\leq}(\mu x, \th)
	$$
	would imply that the pair $T\circ ( T^{-1} \circ K^{\leq} + \Delta_1)$,
	$R^{\leq}$ is a solution of the semiconjugation equation
	$F\circ K = K \circ R$. The map
	$$
	\Delta := T\circ ( T^{-1} \circ K^{\leq} + \Delta_1 ) - K^{\leq} =
	DT (T^{-1} \circ K^{\leq}) \Delta_1 + \mathcal{O}(\|\Delta_1\|^2)
	$$
	belongs to $\X_{Q+1} \times \X_{Q+1} \times \X_Q$ and provides the correction
	to $K^{\leq} $ that makes $F\circ (K^{\leq} + \Delta) = (K^{\leq} + \Delta ) \circ
	R^{\leq}$.
	
	This justifies that from now on we assume that $F$ has the form~\eqref{FinalexpF}.
	\begin{remark}
	As we pointed out along the proof of Lemma~\ref{changevariables}, the parameter $\mu=(\overline{a}(\lambda))^{-\alpha}$
	is well defined if we choose the complex extension of $\Lambda$ to be small enough. Moreover, the scaling $\mu x$ of
	the independent variable $x$ implies a change of the parameters $\beta$ and $\rho$ of the complex sector $S(\beta,\rho)$
	where the function $\Delta$ is defined.
	\end{remark}
To finish this section, we present a result which is a rewording of Lemma 7.1 of~\cite{BFM17}.

\begin{lemma} \label{lemaR}
	Let $R$ be an analytic map in a neighbourhood of the origin of the form
	$
	R(x) = x- ax^N + \O(|x|^{N+1})
	$
	with $a>0$. For $0<\eta<a$, let $\mathcal{R}_\eta :[0,\infty) \to \R$ be defined by
	$$
	\mathcal{R}_\eta(s)=\frac{s}{[1+( a-\eta)(N-1) s^{N-1}]^{\alpha}},\qquad \alpha=\frac{1}{N-1}.
	$$
	Then, for any $0<\eta< a$, there exists $\beta,\rho>0$ such that $R$ maps $S(\beta,\rho)$ into itself and
	its $k$-th iterate satisfies
	$$
	|R^k(x)| \leq \mathcal{R}^k_\eta(|x|)=\frac{|x|}{[1+k( a-\eta)(N-1) |x|^{N-1}]^{\alpha}},\qquad x\in S(\beta,\rho), \quad k\geq 0.
	$$
\end{lemma}
\begin{remark}\label{RemarkR}
If $a$ is a real analytic function on $\lambda\in \Lambda$, being $\Lambda$ relatively compact
and satisfying that $a(\lambda)>0$ on $\Lambda$, it can be proven that there exists
an open set $\Lambda_{\C}' \subset \C^{p}$ such that
	$$
	|R^k(x)| \leq \mathcal{R}^k_\eta(|x|)=\frac{|x|}{[1+k(|a(\lambda)|-\eta)(N-1) |x|^{N-1}]^{\alpha}},\qquad x\in S(\beta,\rho), \quad k\geq 0.
	$$
	
	Indeed, to prove this remark, we only need to apply Lemma~\ref{lemaR} to $\widetilde{R}(x)=\mu R(\mu^{-1} x)$ with $\mu=\big( a(\lambda))^{-\alpha}$.
\end{remark}

	\subsubsection{Invertibility of an auxiliary linear operator}
	Let
	$$
	M(x,\theta) = \left (\begin{array}{ccc} 1 & 0 & 0 \\
	0 & \Id + (K_{x}^{\leq} (x,\th))^{N-1} \J & 0 \\
	0 & 0 & \Id
	\end{array}\right ).
	$$
	We introduce the linear operator
	$$
	\L\Delta = M\Delta - \Delta \circ R^\leq
	$$
	and we rewrite the condition~\eqref{invcondsection} as:
	$$
	\L \Delta = - (F\circ K^{\leq} - K^{\leq} \circ R^\leq) - (F\circ (K^{\leq} + \Delta ) - F\circ K^\leq - M \Delta).
	$$
We introduce the operator $\mathcal{E}$
	\begin{equation}\label{defT}
	\mathcal{E} (\Delta) = F\circ (K^{\leq} + \Delta ) - F\circ K^\leq - M \Delta
	\end{equation}
and we recall the definition of $\Ein=F\circ K^{\leq} - K^{\leq} \circ R^\leq$ in~\eqref{defEleqsec}.
	To solve the invariance condition~\eqref{invcondsection}, we will deal with the equivalent fixed point equation
	\begin{equation}\label{defOpGmpa}
	\Delta=\GGG (\Delta ):= -\L^{-1} \Ein - \L^{-1} \E(\Delta).
	\end{equation}
	For that we have to study the invertibility of $\L$ and to obtain bounds of $\|\L^{-1}\|$.
	
	We have
$$
	(\L \Delta)_{x,\th} = \Delta_{x,\th} - \Delta_{x,\th}\circ R^\leq, \qquad
	(\L \Delta)_y = (\Id + (K_{x}^{\leq} (x,\th))^{N-1} \J)\Delta_y - \Delta_y \circ R^\leq.
$$
	The estimates for $\L$ and $\L^{-1}$ will follow from the next lemma applied to each component of $\L$ working in the appropriate space
	$\X_q$ with either $\J=0$ or $\J\neq 0$.
	
	\begin{lemma}
		Let $q\geq N\ge 2$, $m\ge 1$, $a>0$, $\omega\in \R^d$,
		$R:S(\beta_0,\rho_0)\times \TT^d_{\sigma'} \times \Lambda_\C'\to S(\beta_0,\rho_0)\times \TT^d_{\sigma'}$ of the form
		$
		R(x,\th) = (R_x(x,\th), \th + \omega)
		$
		with $R_x(x,\th)=x- a x^N + \O(|x|^{N-1})$ uniformly in $(\theta,\lambda)$
		and
		$
		\kappa:S(\beta_0,\rho_0)\times \TT^d_{\sigma'} \times \Lambda_\C'\to \C
		$
		satisfying $|\kappa (x,\th)-x| \le C|x|^2$ for some constant $C$.
		
		Let $B:\Lambda_\C' \to \in L(\R^m, \R^m)$ be real analytic such that either $\Re \Spec(B) >0$ or $B=0$ and
		$L:\X_{q} \to \X_{q}$ be the operator defined by
		$$
		L\Delta =(\Id + \kappa^{N-1} B)\Delta - \Delta\circ R.
		$$
		
		Then,
		\begin{enumerate}
			\item[(1)] $L$ is a bounded operator and $\|L\| \leq 2 + C'\rho^{N-1}$ for some $C'>0$.
			\item[(2)] If $B$ is close enough to a diagonal matrix, then given $0<\eta<a$ there exist $\beta, \rho>0$ such that  $L$ has a right inverse
			$\S:\X_{q+N-1} \to \X_{q}$ acting on functions with domain $S(\beta,\rho) \times \TT_{\sigma'}\times \Lambda_\C'$, and
			$$
			\|\S\|\leq \frac{1}{(a-\eta)q} + \rho^{N-1}.
			$$
		\end{enumerate}
	\end{lemma}
	
	\begin{proof} (1) follows directly from the definition of $L$. To prove (2) we first note that an expression for $\S$ is given by
		$$
		\S H = \sum_{j=0}^\infty \big [ \Id + \kappa^{N-1} B\big ]^{-1} \cdots \, [ \Id + (\kappa \circ R^j)^{N-1}  B\big ]^{-1} H\circ R^j .
		$$
		By Lemma \ref{lemaR}, the images of the iterates $R^j$ belong to the domain of $\kappa$.
		When $B\neq 0$, the eigenvalues of $\Id + (\kappa \circ R^j)^{N-1}  B$ are $1+ (\kappa \circ R^j)^{N-1} \mu$ with $\mu \in \Spec B$.
		The quantity $(\kappa \circ R^j)^{N-1}$ belongs to $\kappa (S(\beta,\rho)) \subset S(\beta', \rho')$ with $\beta' = \beta + \O(\rho)$
		and $\rho'= \rho + \O(\rho ^2)$. Since $\Re \mu>0$ and $B$ is as close as we need to a diagonal matrix, for all $v\in \R^m$,
		$
		\big \|\big [\Id + (\kappa \circ R^j)^{N-1}  B\big ] v\big \| > \|v\|
		$
		which implies
		$$
		\big \|\big [\Id + (\kappa \circ R^j)^{N-1}  B\big ]^{-1} \big\| \leq  1,\qquad j\geq 0.
		$$
		Then in both cases, $B=0$ and $B\neq 0$,  under our hypotheses,
		\begin{align*}
		\|\S H(x,\th)\| \leq &\sum_{j=0}^{\infty} \|H (R^{j}(x,\th))\| \leq  \|H\|_{q+N-1} \sum_{j=0}^{\infty} |R^j_x (x,\th)|^{q+N-1} \\
		\leq &  \|H\|_{q+N-1} \sum_{j=0}^{\infty} \frac{|x|^{q+N-1}}{ \big ( 1 +j(a-\eta) (N-1) |x|^{N-1}\big )^{\alpha(q+N-1)}} \\
		\leq & \|H\|_{q+N-1} |x|^{q+N-1} \left  ( 1 + \int_{0}^\infty \frac{ds} { \big ( 1 +s(a-\eta) (N-1) |x|^{N-1}\big )^{\alpha(q+N-1)}}\right  ) \\
		\leq &  \|H\|_{q+N-1} |x|^{q+N-1} \left ( 1 +  \frac{\alpha}{(a-\eta)  |x|^{N-1} } \int_{0}^\infty \frac{du} {  ( 1 +u  )^{\alpha(q+N-1)}}\right  )  \\
		\leq &  \|H\|_{q+N-1} |x|^{q} \left  ( |x|^{N-1} +  \frac{1}{(a-\eta) q }\right  )
		\end{align*}
		and hence
		$$
		\|\S H\|_q \leq \left (\rho^{N-1} +  \frac{1}{(a-\eta) q }\right  ) \|H\|_{q+N-1}.
		$$
	\end{proof}
	
\subsubsection{Estimates for the operator $\GGG$ in~\eqref{defOpGmpa}}
	
	Now we introduce the product space
	$\X_{q}^\times = \X_q \times \X_q \times \X_{q-1}$, $q\geq 2,$ with the product
	norm
	$$
	\|K \|^{\times}_q = \max\big \{\|K_x\|_q, \|K_y\|_q, \|K_\th\|_{q-1}\big \}.
	$$
	
	Consider $\E(\Delta)$ defined in~\eqref{defT} as an operator acting on
	$\Delta$ and let $\E_x,\E_y$ and $\E_\th$ be its components. Notice that this operator
	depends, among other things, on the scaling parameter $\e$. Henceforth $C$ will denote a generic constant.
	\begin{lemma}
		Given $\Aee,\e>0$ there exists $\rho>0$ small and $C>0$ such that
		the Lipschitz constants of the operators $\E_x,\E_y:
		{\X}_{Q+1}^\times \to \X_{Q+N}$ and $\E_\th: {\X}_{Q+1}^\times \to \X_{Q+N-1}$
		are bounded by
		$$
		\Lip \E_x \leq N (1+ \Aee) + C(\rho + \e), \qquad
		\Lip \E_y \leq C\rho \e^{-1}, \qquad \Lip \E_\th \leq  C\rho.
		$$
	\end{lemma}
	\begin{proof}
		We take $\Delta, \widetilde{\Delta} \in {\X}_{Q+1}^\times$. Since
		$|\Delta_x(x,\th)|\leq |x|^{Q+1} \|\Delta \|_{Q+1} $ and analogous bound for
		the other components of $\Delta$ and the ones of $\widetilde{\Delta}$, if
		$\rho$ is small, all compositions involved in~\eqref{defT} make sense.
		
		We decompose
$$
		\big (\E(\Delta) - \E(\widetilde{\Delta})\big )_{x,y} = Z_{x,y}^1 + Z_{x,y}^2 + Z_{x,y}^3, \qquad
		\big (\E(\Delta) - \E(\widetilde{\Delta})\big )_\th = Z_\th^1 + Z_\th^2
$$
		with		
		\begin{align*}
		Z_x^1 =&-(K_x^{\leq} + \Delta_x)^N + (K_x^{\leq} + \widetilde{\Delta}_x)^N, \\
		Z_x^2 =& \f_N( K^{\leq} + \Delta)- \f_N( K^{\leq} + \widetilde{\Delta}), \;\;\;\;
		Z_x^3 = \f_{\geq N+1} ( K^{\leq} + \Delta)- \f_{\geq N+1}( K^{\leq} +
		\widetilde{\Delta}), \\
		Z_y^1 =& (K_x^{\leq} + \Delta_x)^{N-1} \J (K_y^{\leq} + \Delta_y) -
		(K_x^{\leq} + \widetilde{\Delta}_x)^{N-1} \J (K_y^{\leq} + \widetilde{\Delta}_y) \\ &- (K_x^\leq)^{N-1} \J (\Delta_y - \widetilde{\Delta}_y),\\
		Z_y^2 =& \g_N( K^{\leq} + \Delta)- \g_N( K^{\leq} + \widetilde{\Delta}), \;\;\;\;
		Z_y^3 = \g_{\geq N+1} ( K^{\leq} + \Delta)- \g_{\geq N+1}( K^{\leq} + 		\widetilde{\Delta}), \\
		Z_\th^1 =& \h_N( K^{\leq} + \Delta)- \h_N( K^{\leq} + \widetilde{\Delta}), \;\;\;\,
		Z_\th^2 =\h_{\geq N+1} ( K^{\leq} + \Delta)- \h_{\geq N+1}( K^{\leq} + \widetilde{\Delta}).
		\end{align*}
We assume that $\Delta, \widetilde{\Delta}$ belong to a ball of radius $r$ (to
		be determined later) in ${\X}_{Q+1}^\times$. Then
		$$
		Z_x^1 = - N \int_{0}^1 \big [ K_x^{\leq} + \widetilde{\Delta}_x +
		s(\Delta_x - \widetilde{\Delta}_x)\big ]^{N-1} ( \Delta_x - \widetilde{\Delta}_x)
		\, ds
		$$
		and, since $|K_x^{\leq} + \widetilde{\Delta}_x +
		s(\Delta_x - \widetilde{\Delta}_x)|\leq |x| (1+ C_1 \rho)$ with $C_1$ independent of $\th$,
		\begin{align*}
		|Z_x^1 | &\leq  N \big [|x| (1+ C_1\rho)\big ]^{N-1}  \|\Delta_x - \widetilde{\Delta}_x\|_{Q+1} |x|^{Q+1} \\
		& \leq N (1+ \Aee) \|\Delta_x - \widetilde{\Delta}_x\|_{Q+1} |x|^{Q+N}
		\end{align*}
		with $\Aee = \mathcal{O} (\rho)$. Concerning $Z_x^2$,
		$$
		Z_x^2 = \int_{0}^1 \big [ \partial_x \f_N\, (\Delta_x-\widetilde{\Delta}_x )
		+\partial_y \f_N\, (\Delta_y-\widetilde{\Delta}_y )
		+\partial_\th \f_N\, (\Delta_\th-\widetilde{\Delta}_\th ) \big ]\, ds,
		$$
		where the partial derivatives are evaluated at $K^{\leq} + \widetilde{\Delta} +
		s(\Delta-\widetilde{\Delta})$. Then
		\begin{align*}
		|Z_x^2 |\leq & C|x|^N \|\Delta_x - \widetilde{\Delta}_x \|_{Q+1} |x|^{Q+1}
		+ C|x|^{N+1} \|\Delta_\th - \widetilde{\Delta}_\th\|_{Q} |x|^{Q} \\
		&+\big [\|\f_{N-1,1}(\th) \| |x|^{N-1}  + \mathcal{O}(|x|^N) \big ]
		\cdot \|\Delta_y - \widetilde{\Delta}_y \|_{Q+1} |x|^{Q+1}.
		\end{align*}
		By Lemma~\ref{changevariables}, the term
		$\f_{N-1,1}(\th)$ is of order of the rescaling parameter $\e$. Then,
		$$
		\|Z_x^2 \|_{Q+N} \leq C\rho \|\Delta_x -\widetilde{\Delta}_x \|_{Q+1}
		+ C\e \|\Delta_y - \widetilde{\Delta}_y \|_{Q+1} +
		C\rho \|\Delta_\th - \widetilde{\Delta}_\th \|_{Q}
		$$
		if $\rho\ll\e$. We have
		\begin{align*}
		Z_y^1 =& \big [ (K_x^{\leq} + \Delta_x)^{N-1} - (K_{x}^{\leq})^{N-1}\big ]
		\J (\Delta_y - \widetilde{\Delta}_y) \\
		&+ \big [ (K_x^{\leq} + \Delta_x )^{N-1} -
		(K_x^{\leq} + \widetilde{\Delta}_x)^{N-1}\big ] \J(K_y^{\leq} + \widetilde{\Delta}_y).
		\end{align*}
		Then
		\begin{align*}
		\|Z_y^1 \| \leq &C(N-1) |x|^{N-2} \|\Delta_x\|_{Q+1} |x|^{Q+1} \|\J\|
		\|\Delta_y - \widetilde{\Delta}_y\|_{Q+1} |x|^{Q+1} \\
		&+C(N-1) |x|^{N-2} \|\Delta_x - \widetilde{\Delta}_x\|_{Q+1} |x|^{Q+1}
		\|\J\| C|x|^2
		\end{align*}
		and hence
		$$
		\|Z_y^1 \|_{Q+N} \leq C\rho^Q \|\J\| \|\Delta_x \|_{Q+1}
		\|\Delta_y - \widetilde{\Delta}_y \|_{Q+1} +
		C \rho \|\J\|  \|\Delta_x - \widetilde{\Delta}_x\|_{Q+1}.
		$$
		The remaining terms are bounded in the same way as for $Z_x^2$. We obtain
		\begin{align*}
		\|Z_x^3 \|_{Q+N} , \|Z_\th^1 \|_{Q+N-1} &\leq
		C\rho \big [\|\Delta_x -\widetilde{\Delta}_x \|_{Q+1} +
		\|\Delta_y-\widetilde{\Delta}_y \|_{Q+1}
		+ \|\Delta_\th - \widetilde{\Delta}_\th \|_{Q} \big ], \\
		\|Z_y^2 \|_{Q+N} &\leq
		C\rho \e \big [\rho \|\Delta_x -\widetilde{\Delta}_x \|_{Q+1} +
		 \|\Delta_y-\widetilde{\Delta}_y \|_{Q+1}
		+ \rho\|\Delta_\th - \widetilde{\Delta}_\th \|_{Q} \big ], \\
		\|Z_\th^2 \|_{Q+N-1} &\leq
		C\rho^2 \big [\|\Delta_x -\widetilde{\Delta}_x \|_{Q+1} +
		\|\Delta_y-\widetilde{\Delta}_y \|_{Q+1}
		+ \|\Delta_\th - \widetilde{\Delta}_\th \|_{Q} \big ].
		\end{align*}
		However, $Z^3_y$ is a little bit special as we pointed out in Remark~\ref{remarkgN+1}. For
		it we have
		$$
		\| Z_y^3 \|_{Q+N} \leq C \rho \e^{-1}
		\big [\|\Delta_x -\widetilde{\Delta}_x \|_{Q+1} +
		\e \|\Delta_y-\widetilde{\Delta}_y \|_{Q+1}
		+ \|\Delta_\th - \widetilde{\Delta}_\th \|_{Q} \big ].
		$$
	\end{proof}
	
The proof of Theorem~\ref{thetruemanifold} follows immediately from the next lemma
	and the fixed point theorem.
	\begin{lemma}
		There exists $r>0$ such that $\GGG$ defined in~\eqref{defOpGmpa} sends the closed ball
		$\overline{B}(0,r)\subset {\X}_{Q+1}^\times $ into itself and is a contraction on it.
	\end{lemma}
	\begin{proof}
		Let $r_0=\|\L^{-1} \Ein\|_{Q+1}^{\times}$. Given $r>0$ let
		$\Delta,\widetilde{\Delta}\in \overline B(0,r)$. We are going to estimate
		$
		\L^{-1} \E(\Delta) - \L^{-1} \E(\widetilde{\Delta}) \in {\X}_{Q+1}^\times.
		$
		We estimate each component:
		\begin{align*}
		\|\L^{-1}_x (\E(\Delta) -\E(\widetilde{\Delta}))_x \|_{Q+1} \leq &
		 \left (\frac{1}{(1-\eta)(Q+1)} + \rho^{N+1} \right ) \\
		&\times \big [N (1+\Aee) + C (\rho + \e)\big ] \|\Delta-\widetilde{\Delta} \|_{Q+1}^{\times}, \\
		\|\L^{-1}_y (\E(\Delta) - \E(\widetilde{\Delta}))_y \|_{Q+1}
		\leq &\left (\frac{1}{(1-\eta)(Q+1)} + \rho^{N+1} \right )
		C \rho\e^{-1}  \|\Delta-\widetilde{\Delta} \|_{Q+1}^{\times}, \\
		\|\L^{-1}_\th (\E(\Delta) - \E(\widetilde{\Delta}))_\th \|_{Q}
		 \leq & \left (\frac{1}{(1-\eta)Q} + \rho^{N} \right )
		C \rho \|\Delta-\widetilde{\Delta} \|_{Q+1}^{\times}.
		\end{align*}
		Then, since $Q+1>N$, there exist $\rho,\e>0$ small enough such that
		$\rho \e^{-1} $ is small and we can choose $\eta, \Aee$ so small that
		\begin{equation}\label{lipschitz}
		\|\L^{-1} (\E(\Delta) - \E(\widetilde{\Delta}))\|_{Q+1}^{\times} \leq \gamma
		\|\Delta - \widetilde{\Delta}\|_{Q+1}^{\times}
		\end{equation}
		with $\gamma<1$. We choose $r>0$ such that $r_0+\gamma r\leq r$ and then if
		$\Delta \in \overline B(0,r)$, since $\GGG(0) = \L^{-1} \Ein$,
		\begin{align*}
		\|\GGG(\Delta)\|_{Q+1}^{\times} &\leq \|\L^{-1} \Ein \|_{Q+1}^{\times}+
		\|\L^{-1} (\E(\Delta)- \E(0))\| \leq r_0 + \gamma\|\Delta- 0 \|_{Q+1}^{\times} \\
		& \leq r_0 + \gamma r \leq r,
		\end{align*}
		which proves that $\GGG$ sends the ball $\overline B(0,r)$ into itself.
		Moreover,~\eqref{lipschitz} directly implies that $\GGG$ is a contraction.
	\end{proof}


\subsection{Formal parabolic manifold. Proof of Theorem~\ref{formalmanifold}}\label{sec:proofsmapsformal}
This section is devoted to the computation of a formal approximation of a solution of the
semiconjugation condition $F\circ K=K\circ R$ when $F$ is of the form \eqref{system}.
The solution certainly is not unique.  We have chosen a structure for the terms which
appear in the approximation. There is a lot of freedom for obtaining the terms of $K$ and $R$.
This freedom is seen when solving the cohomological equations at each  order.
Our main motivation has been to show that such approximation actually exists and is  computable.
In this section we admit $P\geq 1$.
	
	We prove by induction over $j$ that there exist $K^{(j)	}$ and $R^{(j)}$. Assuming the form \eqref{formkx}, \eqref{formky}, \eqref{formkth}
	for $K^{(j)}_x$, $K^{(j)}_y$, $K^{(j)}_\th$ respectively, the form
	$R^{(j)}_x(x,\theta)=  x+ \sum_{l=1}^j R^{l+N-1}_{x}(\theta) x^{l+N-1}$ and the form \eqref{formrth} for $R^{(j)}_\th$, we will prove that at step $j$
	we are able to determine the quantities
	$\overline K^j_{x,y}$, $\overline K^{j-1}_{\th}$, $\widetilde K^{	j+N-1}_{x,y}$,
	$\widetilde K^{j+P-2}_{\th}$, $ R^{j+N-1}_{x}(\theta)$
	and $ R^{j+P-2}_{\th}(\theta)$ so that the order condition
	\eqref{ordreE} for the remainder $E^{(j)}$ is fullfilled.
	
	Let us first assume that $P\leq N$. We deal with first step of the induction procedure, $j=1$. We write
	\begin{align*}
	K^{(1)}_{x}(x,\theta) & = x+
	\widetilde K^{N}_{x}(\theta) x^{N}, \qquad
	K^{(1)}_{y}(x,\theta)   = 0, \qquad
	K^{(1)}_{\theta}(x,\theta)   = \th,   \\
	&R^{(1)}_x(x,\theta ) = x+ R^N_x(\th) x^N,  \qquad
	R^{(1)}_\theta (x,\theta) = \theta +\omega,
	\end{align*}	
	and we compute $E^{(1)} = F\circ K^{(1)}- K^{(1)}\circ R^{(1)}$. From the form~\eqref{system} of $F$ we
	obtain
	\begin{align*}
	E^{(1)}_{x}(x,\theta)   = & [\widetilde K_x^N(\th) -\widetilde K_x^N(\th+\omega )-R^N_x(\th)-a(\th)] x^N+ \O(|x|^{N+1}), \\
	E^{(1)}_{y}(x,\theta)   =& \O(|x|^{N+1}) , \\
	E^{(1)}_{\th}(x,\theta)   = & \O(|x|^{P}).
	\end{align*}	
	To have $E^{(1)}_{x} (x,\theta)= \O(|x|^{N+1})$ we take
	$$
	\overline R^N_x=-\overline a, \qquad \widetilde R^N_x(\th)= 0, \qquad \widetilde K^N_x(\th)= -\SD \big(\widetilde  a\big) (\th).
$$
	For $j\ge 2$, assuming the induction hypothesis, we write $K^{(j)} = K^{(j-1)} + \mathcal{K}^{(j)}$ and $R^{(j)} = R^{(j-1)} + \mathcal{R}^{(j)}$
	with $E^{(j-1)}=F\circ K^{(j-1)}-K^{(j-1)}\circ R^{(j-1)}$ satisfying
	\begin{equation}\label{Ej-1ordre}
	\begin{aligned}
	E^{(j-1)} = &(E_x^{j+N-1}(\th) x^{j+N-1} , E_y^{j+N-1}(\th) x^{j+N-1} , E_\th^{j+P-2}(\th) x^{j+P-2}) \\
	&+(\O(|x|^{j+N}),\O(|x|^{j+N}) ,\O(|x|^{j+P-1}) )
	\end{aligned}
	\end{equation}
	and $\mathcal{K}^{(j)}$, $\mathcal{R}^{(j)}$ of the form:
	$$
	\mathcal{K}^{(j)} =
	\begin{pmatrix}
	\overline K^{j}_{x} x^{j}+ \widetilde K^{j+N-1}_{x}(\theta) x^{j+N-1}  \\
	\overline K^{j}_{y} x^{j}+ \widetilde K^{j+N-1}_{y}(\theta) x^{j+N-1}  \\
	\overline K^{j-1}_{\th}  x^{j-1}+	\widetilde K^{j+P-2}_{\th}(\theta) x^{j+P-2}
	\end{pmatrix},\qquad
	\mathcal{R}^{(j)}=
	\begin{pmatrix}
	R^{j+N-1}_x (\th) x^{j+N-1}  \\
	R^{j+P-2}_\theta  (\th) x^{j+P-2}
	\end{pmatrix}.
	$$

The error term at the step $j$, ${E}^{(j)}=F\circ K^{(j)}-K^{(j)}\circ R^{(j)}$, is decomposed as
	\begin{align*}
	{E}^{(j)} = & {E}^{(j-1)} + \big [F\circ{K}^{(j)} - F\circ{K}^{(j-1)} -(DF\circ {K}^{(j-1)} ) {\mathcal K}^{(j)}\big ] \\
	&+ (DF\circ{K}^{(j-1)} ) {\mathcal K}^{(j)} -{\mathcal K}^{(j)} \circ {R}^{(j-1)} \\
	& -\big [{K}^{(j)}\circ {R}^{(j)} -{K}^{(j)} \circ {R}^{(j-1)}\big ].
	\end{align*}
We first compute the terms in $E^{(j)}_{x,y}$ that are of order less than $\O(|x|^{j+N})$ and
the terms in $E^{(j)}_{\th}$ of order less than $\O(|x|^{j+P-1})$.
By~\eqref{Ej-1ordre} we are done with the term $E^{(j-1)}$. To proceed with the other terms we use Taylor's theorem, that
$K^{(j-1)}(x)=(x,0,\th+\O(|x|))+\O(|x|^2)$, $R^{(j-1)}_x(x,\th) = x-ax^N+ \O(|x|^{N+1})$, $R^{(j-1)}_\th=\th + \omega+\O(|x|)$ and that
$F$ has the form~\eqref{system} together with the forms of $\mathcal{K}^{(j)}$
and $\mathcal{R}^{(j)}$.

	By Taylor's theorem we have that
	$$
[F\circ{K}^{(j)} - F\circ{K}^{(j-1)} -(DF\circ {K}^{(j-1)} ) {\mathcal K}^{(j)}\big ]= (\O(|x|^{j+N}),\O(|x|^{j+N}), \O(|x|^{j+P-1}) ).
	$$
	The computations have to be done carefully, considering the cases $P=1$ and $P\ge 2$ separately.
	
	Concerning $(DF\circ{K}^{(j-1)} ) {\mathcal K}^{(j)}$, 	
	\begin{equation*}
	(DF\circ K^{(j-1)} ){\mathcal K}^{(j)}=
	\left (
	\begin{array}{ccc}
	1- N a(\th) x^{N-1}
	& f_{N-1,1}(\th) x^{N-1}
	& -\partial_\th a(\th) x^N\\
0& \Id + x^{N-1} B(\theta) &
0 \\
0 & 0 & \Id
	\end{array}
	\right ) \mathcal{K}^{(j)} + \widecheck{e}_1
	\end{equation*}
with $\widecheck{e}_1=(\O(|x|^{j+N}),\O(|x|^{j+N}), \O(|x|^{j+P-1}) )$. Then, since
$\widetilde{\mathcal{K}}^{(j)}_{x,y}=\O(|x|^{j+N-1})$,
	\begin{align*}
(DF\circ K^{(j-1)} )&{\mathcal K}^{(j)}-\mathcal{K}^{(j)} =\\
	&\left (
	\begin{array}{c}
- Na(\th) x^{N-1} \overline{\mathcal{K}}^{(j)}_x
	+f_{N-1,1}(\th) x^{N-1} \overline{\mathcal{K}}^{(j)}_y
	-\partial_\th a(\th) x^{N} \mathcal{K}_\th^{(j)}
	\\ x^{N-1} B(\theta) \overline{\mathcal{K}}^{(j)}_y
	\\ 0
	\end{array}
	\right ) +
{e}_1
	\end{align*}
with $e_1=(\O(|x|^{j+N}),\O(|x|^{j+N}), \O(|x|^{j+P-1}) )$. In addition, by Taylor's theorem,
	\begin{equation*}
	\mathcal{K}^{(j)}\circ R^{(j-1)} = \mathcal{K}^{(j)}(x,\th+\omega) -\overline{a} x^N \partial_x \overline{\mathcal{K}}^{(j)} (x) + e_2
	\end{equation*}
	with $e_2=(\O(|x|^{j+N}),\O(|x|^{j+N}), \O(|x|^{j+P-1}) )$.
	
	Concerning $\eta^{(j)}:=-\big [{K}^{(j)}\circ {R}^{(j)} -{K}^{(j)} \circ {R}^{(j-1)}\big ]$, we write it as
	$$
-\int_0^1 \big(\partial_x K^{(j)}\cdot  \mathcal{R}^{(j)}_x(\th) + \partial_\th K^{(j)} \cdot \mathcal{R}^{j}_\th(\th) \big)  \, ds,
	$$
	where $\partial_x K^{(j)}, \partial_y K^{(j)}$ are evaluated at $R^{(j-1)}+s {\mathcal R}^{(j)}$.
	The computation gives
	$$
	\eta^{(j)}=
	-\begin{pmatrix}
	R^{j+N-1}_{x}(\th) x^{j+N-1}
	+\partial _\th \widetilde K^{N}_x(\th+\omega )R^{j+P-2}_\th x^{j+N+P-2}\\
0 \\
	R^{j+P-2}_{\th}(\th)  x^{j+P-2}
	\end{pmatrix}
	+e_3
	$$
	with $e_3=(\O(|x|^{j+N}),\O(|x|^{j+N}), \O(|x|^{j+P-1}) )$.
	
	From these computations we obtain
	\begin{align}
{E}^{(j)}_{x}(x,\theta)   = & \big [\widetilde K_x^{j+N-1}(\th) -\widetilde K_x^{j+N-1}(\th+\omega )+\big(j\overline a -N  a(\th)\big)\overline K_x^j -
	R^{j+N-1}_x(\th) \nonumber \\
	& -\partial_\th a(\th)  \overline K^{j-1}_\th + f_{N-1,1}(\th)  \overline K^{j}_{y}+E_x^{j+N-1}(\th) \big] x^{j+N-1} +\O(|x|^{j+N})\nonumber\\
	& - \big[\partial_\th \widetilde{K}_x^N (\th + \omega) R_\th^{j+P-2} +\partial_\th a (\th) \widetilde{K}_\th^{j+P-2}\big] x^{j+N+P-2}	
	\label{Ejx}\\
{E}^{(j)}_{y}(x,\theta)   =& \big[\widetilde K_y^{j+N-1}(\th) -\widetilde K_y^{j+N-1}(\th+\omega)
	+ \big(B(\th) + j\overline  a\Id \big) \overline K^{j}_{y}\nonumber \\ & +  E_y^{j+N-1}(\th) \big] x^{j+N-1}
 +\O(|x|^{j+N})\label{Ejy}
	\\
{E}^{(j)}_{\th}(x,\theta)   = & \big[\widetilde K_\th^{j+P-2}(\th) -\widetilde K_\th^{j+P-2}(\th+\omega) -R^{j+P-2}_\th(\th)
	+  E_\th^{j+P-2}(\th)\big]x^{j+P-2}\notag \\ &+(j-1) \overline a \overline K_\th^{j-1} x^{j+N-2} +\O(|x|^{j+P-1}).\label{Ejth}
	\end{align}
	The condition on the order $E^{(j)}$, namely~\eqref{Ej-1ordre} for $j$, provides the so-called cohomological equations in this setting. 	
	Next we solve them distinguishing cases when necessary and trying to keep $R$ as simple as possible, namely, taking the value
	$0$ for $\mathcal{R}^{(j)}$ if it is possible.
	
	We start with \eqref{Ejy}. We take
	$$
	\overline K^{j}_{y}= -[\overline B + j \overline a\Id]^{-1} \overline E_y^{j+N-1} , \qquad
	\widetilde K_y^{j+N-1} = \SD \big(\widetilde B  \cdot \overline K^{j}_{y}+ \widetilde  E_y^{j+N-1}\big).
	$$
	Then from  \eqref{Ejth}, when $P<N$
	$$
	R^{j+P-2}_\th =\overline R^{j+P-2}_\th=\overline{E}^{j+P-2}_\th ,  \qquad \overline K^{j-1}_\th \; \text{free},
	\qquad
	\widetilde K^{j+P-2}_\th =\SD \big(\widetilde E_\th^{j+P-2}\big)
	$$
	and if $P=N$
	$$
	R^{j+P-2}_\th =0,\qquad \overline K^{j-1}_\th = -\frac{\overline{E}_\th^{j+P-2} }{(j-1) \overline a},
	\qquad \widetilde K^{j+P-2}_\th=\SD \big(\widetilde E_\th^{j+P-2} \big).
	$$
	Finally, we deal with \eqref{Ejx}. For that we introduce the already known functions
	\begin{align*}
	\varphi^{(j)}(\th)=&
	-\partial_\th a(\th) \overline{K}_\th^{j-1} + f_{N-1,1}(\th) \overline{K}_y^j +E_x^{j+N-1} (\th) \\
	\psi^{(j)}(\th)= & \begin{cases}
	\varphi^{(j)}(\th), & P\neq 1, \\
	\varphi^{(j)}(\th) - \partial_\th \widetilde{K}^N(\th+\omega) R_\th^{j+P-2} - \partial_\th a(\th) \widetilde{K}^{j+P-2}_\th(\th), & P=1,
	\end{cases}
	\end{align*}
	and we notice that we have to solve
	$$
	\widetilde K_x^{j+N-1}(\th) -\widetilde K_x^{j+N-1}(\th+\omega )+\big(j\overline a -N  a(\th)\big)\overline K_x^j -
	R^{j+N-1}_x(\th) = \psi^{(j)}(\th).
	$$
If $j = N$ we take
	$$
	R^{j+N-1}_x = \overline R^{j+N-1}_x =  \overline{\psi}^{(j)} ,
	\qquad
	\overline K^j_x \quad \text{free}, \qquad
\widetilde K_x^{j+N-1}=  \SD \big(\widetilde{\psi}^{(j)} -N\widetilde{a} \overline K^j_x  \big)
	$$
	and when $j\neq N$,
	$$
	R^{j+N-1}_x = 0,
	\qquad  \overline K_x^{j}= \frac{\overline{\psi}^{(j)}}{(j-N) \overline a},
	\qquad \widetilde K_x^{j+N-1} = \SD \big (\widetilde{\psi}^{(j)} - N\widetilde{a} \overline{K}_x^j\big ).
	$$
	In this way we have proven that we can always obtain $\mathcal{K}^{(j)}$ and $\mathcal{R}^{(j)}$ such that~\eqref{ordreE} is satisfied.

It only remains to discuss about the case $P>N$. In this case we simply notice that we always can take $P=N$ and $h_P\equiv 0$.
Notice that when $P\geq N$, we can take $R^{(j)}_\th= \th + \omega$ for any $j\in \N$.

\subsection{The stable manifold of the invariant torus. Proof of Corollary~\ref{mapcorollary}}\label{sec:proofcorollary}

The existence of $K$ and $R$ satisfying the invariance condition $F\circ K-K\circ R=0$ and~\eqref{formKR} is straightforwardly guaranteed by
Theorems~\ref{thetruemanifold} and~\ref{formalmanifold}.

To check that $K$ is $\mathcal{C}^\infty$ on $[0,\rho) \times \TT^d \times \Lambda$, we first note that, if $h$ is an analytic function in the sector $S$ such that
$h=\O(|x|^{M})$, then, for  $t\in \R \cap S$, we have that its $l$-derivative satisfies
$\partial_x^{l} h = \O(|x|^{M-l})$. This property is a direct consequence of the geometry of the set $S$ and Cauchy's theorem.

Take $j=N$ and let $K^{(N)}$ and $R^{(N)}$ be given by Theorem~\ref{formalmanifold}. Let $\mathcal{U}_{\C} \times \TT^d_\sigma \times \Lambda_\C$ be a complex domain to which
$F$ has an analytic extension.
Applying Theorem~\ref{thetruemanifold} we obtain that there exists a sector
$S^{(N)} = S(\beta_N,\rho_N)$ and an analytic function
$\Delta^{(N)} = \O(|x|^{N+1})$ defined in $S^{(N)}\times \TT^d_\sigma\times \Lambda_\C$
and satisfying  $F^{(N)}\circ (K^{(N)}+ \Delta^{(N)})-(K^{(N)}+ \Delta^{(N)})\circ R^{(N)}=0$.
Then, we have that for $x\in \R\cap S^j$
$$
\partial_x^l \Delta^{(N)} = \O(|x|^{N+1-l}).
$$
As a consequence the parameterization $K^{(N)}+\Delta^{(N)}$ is $\mathcal{C}^N$ on $[0,\rho_N) \times \TT^d \times \Lambda$.
Now we consider $j>N$ and, applying again Theorems~\ref{formalmanifold} and~\ref{thetruemanifold} in the same way as before, we obtain
$K^{(j)} + \Delta^{(j)}$ is $\mathcal{C}^{j}$ on $[0,\rho_{j})\times \TT^d \times \Lambda$.
Here we also use $R=R^{(N)} $.

As we pointed out in Theorem~\ref{formalmanifold}, $K^{(j)}- K^{(N)} = \O(|x|^{N+1})$. Then, by the uniqueness
of $\Delta^{(j)}$, we have that $\Delta^{(N)} = K^{(j)}- K^{(N)} + \Delta^{(j)}$. Therefore
$K:=K^{(N)}+\Delta^{(N)}=K^{(j)}+\Delta^{(j)}$ is $\mathcal{C}^{j}$ on $[0,\rho_{j})\times \TT^d \times \Lambda $ and $\mathcal{C}^{N}$
at $[0,\rho_{N})\times \TT^d \times \Lambda$. If $\rho_{N}<\rho_j$ we are done. Assume then that $\rho_{N}>\rho_j$.
Since $\overline{a}(\lambda)>0$, there exists $k>0$ such that
$R^k_{t} ([0,\rho_N),\theta,\lambda) \subset [0,\rho_{j})$. Then, from the invariance equation we have that
$$
K = F^{-k}\circ K \circ R^k,
$$
and therefore we can extend the domain of $K$ from $[0,\rho_j) \times \TT^d \times \Lambda$ to
$[0,\rho_N)\times \TT^d \times \Lambda$. We conclude then that for all $j$, $K$ is $\mathcal{C}^j$ at the
domain $[0,\rho_{N})\times \TT^d \times \Lambda$ and the result is proven.

The property $W^{\rm s}_{\rho} = K([0, \rho))$ can be proven using the same geometric arguments as the ones in~\cite{BH08}. We omit the proof.



\section{Proof of the results. Flow case}\label{sec:proofsedos}
We will deduce the \textit{a posteriori} result about the parabolic stable manifold (Theorem~\ref{main_theorem_flow_case}) from the corresponding result
for maps by means of an adequate ostroboscopic map. However,
the result about the approximation of the parabolic manifold (Theorem~\ref{formal_part_theorem_flows})
will be proven directly. The reason is to provide an algorithm to compute such approximation avoiding the calculation of the stroboscopic map,
which would involve the Taylor expansions of the flow around the origin.

We begin in Section~\ref{sec:sdeqvf} reminding key facts on the small divisors equation we will encounter in the vector field setting.
In Section~\ref{sec:proofedostrue} and~\ref{sec:proofsedosformal} we will prove Theorems~\ref{main_theorem_flow_case} and~\ref{formal_part_theorem_flows} respectively.

As we did in Section~\ref{sec:notproofs} we omit the parameters $\beta,\rho$ in $S$ and the dependence on $\lambda$ of our notation.

\subsection{Small divisors equation}\label{sec:sdeqvf}
In the setting of differential equations, the small divisors equation is
\begin{equation}
\label{smalldiveqvf}
\partial _\th \varphi(\theta,\lambda) \cdot \omega= h(\th,\lambda),
\end{equation}
with $h:\TT^d \times \Lambda \to \R^k$ and $\omega\in \R^d$.
If $h(\theta,\lambda)=\sum_{k\in \Z^d,\, k\neq 0 } h_k(\lambda) e^{2\pi ik\cdot \th}$ has zero average and $k\cdot \omega\neq 0$ for all $k\neq 0$, equation~\eqref{smalldiveqvf} has a formal solution
\begin{equation*}
\varphi(\theta,\lambda) = \sum_{k\in \Z^d} \varphi_k(\lambda) \text{e}^{2\pi i k \cdot \theta},\qquad
\varphi_k(\lambda) = \frac{h_k(\lambda)}{2\pi i k\cdot \omega},\qquad k\neq 0.
\end{equation*}
Here $ \varphi_0(\lambda)$ is free.
In this case the analytical result reads as Theorem~\ref{thsmalldiv}, using the definition of Diophantine vector for vector fields
in Section~\ref{subsec:notation}.

As a consequence, if $h: \TT_{\sigma}^{d} \times \CS_\sigma \times \Lambda_\C \to \C^k$ is quasiperiodic in $t$ with frequency vector $\nu \in \R^{d'}$,
$(\omega, \nu)\in \R^{d+d'}$
 is Diophantine and has zero average, then, the equation
\begin{equation} \label{smalldiveqvftwo}
(\partial _\th \varphi(\theta,t,\lambda) , \partial_t \varphi(\theta,t,\lambda))\cdot (\omega, 1) = h(\theta,t,\lambda)
\end{equation}
has a unique solution with zero average defined on $\TT^{d}_{\sigma}\times \CS_{\sigma} \times \Lambda_\C$ and bounded in
$\TT^{d}_{\sigma'} \times \CS_{\sigma'}\times \Lambda_\C$ for any $0<\sigma'<\sigma$.
Indeed, since $h(\th, t,\lambda)=\widehat{h}(\th,\nu t,\lambda)$ with $\hat{h} :\TT^{d+d'}\times \Lambda\to \C^k$,
equation~\eqref{smalldiveqvftwo} is equivalent to
$$
(\partial _\th \widehat \varphi(\theta,\tau ,\lambda) , \partial_\tau \widehat \varphi(\theta,\tau ,\lambda))\cdot (\omega, \nu) = \widehat h(\theta,\tau,\lambda).
$$
The vector field version of the small divisors lemma (analogous to Theorem \ref{thsmalldiv}) assures that this equation has a unique $\widehat \varphi:\TT^{d+d'}_\sigma \times \Lambda\to \C^k$ with zero average.
Then $\varphi(\th,t,\lambda) = \widehat{\varphi}(\th,\nu t,\lambda)$ is the unique solution of equation~\eqref{smalldiveqvftwo} with zero average.
We will denote it by $\SD(h)$.
\subsection{Parabolic manifolds for vector fields depending quasi periodically on time}\label{sec:proofedostrue}

The proof of Theorem~\ref{main_theorem_flow_case} is split into three main parts, the first one contains
preliminary reductions, the second one consists in applying Theorem~\ref{thetruemanifold} to the time-$1$ map obtaining a parabolic stable manifold for
this map, finally the third part is to recover Theorem~\ref{main_theorem_flow_case} by seeking the parabolic stable manifold for the vector field $X$.
This strategy is developed in Sections~\ref{sec:true1}, \ref{sec:true2} and \ref{sec:true3} below. It was also used in~\cite{BFM2015a}.

From now on we consider a vector field $X(x,y,\theta,t)$ depending quasi periodically on time, having
the form given in~\eqref{vectorfield_truemanifold} and assume that all the hypotheses in Theorem~\ref{main_theorem_flow_case}
hold true. From now on we will assume $P=N$ since $h_N\equiv 0$ satisfies our conditions.

\subsubsection{Preliminary reductions and notation}\label{sec:true1}
First we rewrite the vector field as an autonomous skew product vector field
\begin{equation}\label{skewvectorfield}
\begin{aligned}
\dot{x}&= -\widehat a(\theta,\tau) x^N  + \widehat f_{N}(x,y,\theta,\tau) + \widehat f_{\geq N+1}(x,y,\theta,\tau) \\
\dot{y} &= x^{N-1} \widehat B(\theta,\tau) y + \widehat g_{N}(x,y,\theta,\tau) + \widehat g_{\geq N+1}(x,y,\theta,\tau) \\
\dot{\theta} &= \omega + \widehat h_N(x,y,\theta,\tau) + \widehat h_{\geq N+1}(x,y,\theta,\tau)\\
\dot{\tau} &=\nu,
\end{aligned}
\end{equation}
where $\widehat a:\TT^d \times \TT^{d'} \to \C$, $a(\th,t) = \widehat a(\th,\nu t)$ and the same for
the other quantities with hat.

We denote by $\widecheck{X}$ the new vector field:
$$
\widecheck{X}(x,y,\th,\tau) = \left (\begin{array}{c} \widehat{X}(x,y,\th,\tau) \\ \nu\end{array} \right ).
$$
We also introduce
$$
\widecheck{K}^{\leq}(x,\theta,\tau) = \left (\begin{array}{c} \widehat K^{\leq}(x,\th,\tau) \\ \tau
\end{array}\right ),\qquad \widecheck{Y}^{\leq}(x,\th,\tau) = \left (\begin{array}{c}
\widehat Y^{\leq}(x,\th,\tau) \\ \nu\end{array} \right )
$$
A straighforward computation shows that with this notation, condition~\eqref{condEleqfluxos} on
$E^{\leq}$ reads
\begin{equation}\label{condEleqfluxos_2}
\widecheck{E}^{\leq} := \widecheck{X}\circ \widecheck{K}^{\leq} -
D\widecheck{K}^{\leq} \widecheck{Y}^\leq =
(\mathcal{O}(|x|^{Q+N}),\mathcal{O}(|x|^{Q+N}),\mathcal{O}(|x|^{Q+N-1}),0),
\end{equation}
where $D=\partial_{x,\th,\tau}$.

Next we average to transform $\widehat a(\theta,\tau)$ to $\overline{a}$ and $\widehat B(\th,\tau)$
to $\overline B$. This is accomplished with two successive elementary changes of variables:
\begin{align*}
T_1(x,y,\th,\tau)&=(x+c_1(\th,\tau)x^N,y,\th,\tau), \\
T_2(x,y,\th,\tau)&=(x,x^{N-1} C_2(\th,\tau) y,\th,\tau).
\end{align*}
The first one transforms the monomial $-\widehat a(\th,t) x^N$ of the first component of the vector
field into
$$
\big [-\widehat a + \partial_\th c_1 \cdot \omega + \partial_\tau c_1 \cdot \nu \big] x^N
$$
while keeps all other monomials of order $N$ invariant.
Recall that we have introduced the notation (Section~\ref{subsec:notation}) of
$\widetilde{h}= \widehat{h} - \overline{h}$ to denote the oscillatory part of a function on a torus.
Then, using the small divisors lemma, we can choose $c_1$ such that
$$
\partial_\th c_1 \cdot \omega + \partial_\tau c_1 \cdot \nu = \widetilde{ a}
$$
and hence the monomial becomes $-\overline{ a} x^N$.

In an analogous way we choose $C_2$ to transform the monomial $x^{N-1} \widehat B(\th,\tau)y$
of the second component of the vector field into $x^{N-1} \overline{B} y$.

\subsubsection{From flows to maps}\label{sec:true2}
Let $\varphi(t;x,y,\th,\tau)$ be the solution of the vector field $\widecheck{X}$ and
$\psi(t;x,\th,\tau)$ the one of the vector field $\widecheck{Y}^{\leq}$. We define the maps
$$
F(x,y,\th,\tau) = \varphi(1;x,y,\th,\tau),\qquad R(x,\th,\tau)=\psi(1;x,\th,\tau).
$$
\begin{lemma}\label{FRflows}
We have that
\begin{enumerate}
\item $F$ is analytic in $\U_\C\times \TT^{d+d'}_\sigma\times \Lambda_\C$ where $\U_\C$ is a neighbourhood of
$(0,0)\in \C^{1+m}$, $(\th,\tau) \in \TT^{d+d'}_\sigma$ and $\Lambda_\C \subset \C^p$ a complex extension of $\Lambda$.
\item $F$ has the form
\begin{equation}\label{formaFfluxos}
F\left (\begin{array}{c} x \\ y \\ \th \\ \tau \end{array}\right )=
\left (\begin{array}{c} x- \overline a x^{N} + \widecheck{f}_N(x,y,\th,\tau) +
\widecheck{f}_{\geq N+1}(x,y,\th,\tau)\\
y + x^{N-1} \overline B y + \widecheck{g}_N(x,y,\th,\tau) + \widecheck{g}_{\geq N+1}(x,y,\th,\tau) \\
\th + \omega + \widecheck{h}_N (x,y,\th,\tau)+ \widecheck{h}_{\geq N+1}(x,y,\th,\tau) \\
\tau + \nu
\end{array}\right).
\end{equation}
\item $R$ has the form
$$
R(x,\th,\tau)=(x-\overline{a} x^N + \O(|x|^{N+1}), \th+\omega, \tau + \nu).
$$
\end{enumerate}
\end{lemma}
\begin{proof}
Let $z=(x,y,\th,\tau)$, $\eta_{\geq N}:= \widecheck{h}_{N} + \widecheck{h}_{\geq N+1}$ and $\phi(t;z):= \varphi(t;z) -\gamma(t)$ where
$$
\gamma(t) = (x,y,\th+ \omega t, \tau+\nu t)^{\top}.
$$
Then, denoting by $\text{Lip} \, X$ the Lipschitz constant of $\widecheck X$ in the domain $\U_\C$,
$$
\|\phi(t;z)\| \leq  \left \|\int_{0}^t \big (  X_x(\gamma(s)) ,
X_y(\gamma(s)) , \eta_{\geq N}(\gamma(s) ),0 \big ) \right \| \,ds+
\int_{0}^t \text{Lip} \,X \,\|\phi(s;z)\|\,ds.
$$
By Gronwall's lemma we get
$
\|\phi(t;z)\|\leq C \|(x,y)\|^N \text{e}^{t \text{Lip}\,X }
$
and hence
\begin{equation}\label{boundvarphi}
\varphi(t;z)=\gamma(t) + \mathcal{O}(\|(x,y)\|^N).
\end{equation}
On the other hand, by Taylor's theorem
\begin{equation}\label{boundvarphi2}
\begin{aligned}
\varphi(t;z) &= \varphi(0;z) + \dot{\varphi}(0;z) t + \int_{0}^t (t-s) \ddot{\varphi}(s;z) \,ds \\
&=z+X(z) t + \int_{0}^t (t-s)DX(\varphi(s;z)) X(\varphi(s;z))\, ds .
\end{aligned}
\end{equation}
By~\eqref{boundvarphi}
$$
\|DX(\varphi(s;z))\| \leq C \|(x,y)\|^{N-1},
\qquad \|X(\varphi(s;z))-(0,0,\omega,\nu)^{\top}\|
\leq C\|(x,y)\|^{N}
$$
and then
$$
DX(\varphi(s;z)) X(\varphi(s;z)) = DX(\varphi(s;z)) \left (\begin{array}{c}
0 \\0 \\\omega \\ \nu\end{array}\right ) + \mathcal{O}(\|(x,y)\|^{2N-1})
=:e.
$$

Since the derivatives $\partial_\th X$ and $\partial_\tau X$ are of order $N$ the first term
$e_1$ in the right hand side contains terms of order $N$. However, since after the
averaging procedure $\overline a$ depends neither on $\th$ nor on $t$, there is not a monomial related
to $x^N$ in the first component of $e$.
Analogously, there is not a monomial related to $x^{N-1} y$ in the second component of $e$.

Taking $t=1$ in~\eqref{boundvarphi2} we get the form~\eqref{formaFfluxos}.

The proof of the third item follows exactly in the same way, just taking into account that $\widecheck{Y}^{\leq}$ has no $y$ component.
\end{proof}
\begin{lemma}\label{lemma:e}
Let $e(t,x,\th,\tau) := \varphi(t;\widecheck{K}^{\leq}(x,\th,\tau))-\widecheck{K}^{\leq}
(\psi(t;x,\th,\tau))$. We have
$$
e(t,x,\th,\tau) = \big (\mathcal{O}(|x|^{Q+N}), \mathcal{O}(|x|^{Q+N}), \mathcal{O}(|x|^{Q+N-1},0)
$$
uniformly for $t\in [0,1]$ and $(\th,\tau)\in \TT^{d+d'}_\sigma$.
\end{lemma}
\begin{proof}
Let $v= (x,\th,\tau)$. From~\eqref{condEleqfluxos_2} we have that
\begin{align*}
e(x,v) = &\int_{0}^t \widecheck{X}(\varphi(s;\widecheck{K}^{\leq}(v)))\,ds -
\int_{0}^t D\widecheck{K}^{\leq}(\psi(s;v))\widecheck{Y}^{\leq}(\psi(s;v)) \,ds \\
=& \int_{0}^t \big [\widecheck{X}(\varphi(s;\widecheck{K}^{\leq}(v))) -
\widecheck{X}(\widecheck{K}^{\leq}(\psi(s;v)))\big]\, ds  + \int_{0}^t \widecheck{E}^{\leq}(\psi(s;v)) \,ds.
\end{align*}

Given $v=(x,\th,\tau)$ fixed, we introduce
$$
\chi(s)=|x|^{-(Q+N)}|e_x(s,v)| + |x|^{-(Q+N)}\|e_y(s,v)\| + |x|^{-(Q+N-1)} \|e_\th(s,v)\|.
$$
On the one hand, by the estimates in the proof of Lemma~\ref{FRflows} and~\eqref{boundvarphi},
$\|K^{\leq}_{x,y}(\psi(s;v))\|\leq \kappa |x|$, $\|\varphi_{x,y}(s;v)\|\leq \kappa |x|$
and $|\psi_x(s;v)| \leq \kappa |x|$ uniformly in $s$ and $v$.
On the other hand, $\|D\widecheck X(u,y,\th,\tau)\| \leq C_1 |u|^{N-1}$ for $(u,y) \in B_\varrho \subset \C^{1+m}$ uniform with respect to $\varrho$ and $\theta,\tau$.
Using these facts and~\eqref{condEleqfluxos_2} we have that
\begin{align*}
\chi(t)\leq & |x|^{-(Q+N)} \int_{0}^t \big | \widecheck{E}_x^{\leq}(\psi(s;v))\big |\, ds
+ |x|^{-(Q+N)} \int_{0}^t \big \| \widecheck{E}_y^{\leq}(\psi(s;v))\big \|\, ds \\
& +|x|^{-(Q+N-1)} \int_{0}^t \big \| \widecheck{E}_\th^{\leq}(\psi(s;v))\big \|\, ds \\
&+C_1 |x|^{-(Q+N)} \int_{0}^t \big |x|^{N-1} [|e_x(s,v)| + \|e_y(s,v)\| +|x| \|e_\th(s,v)\|\big ] \,ds\\
\leq & C+ C_1|x|^{N-1} \int_{0}^t \chi(s) \,ds
\leq C+ C_2 \int_{0}^t \chi(s)\, ds,
\end{align*}
where we have used that $|x|$ is small enough.
By Gronwall's lemma,
$\chi (t) \leq C \text{e}^{C_2  t}$, for $0\leq t \leq 1$,
and from this inequality we obtain the statement.
\end{proof}
\begin{remark}
Note that Lemmas~\ref{FRflows} and~\ref{lemma:e} provide the hypotheses stated in Theorem~\ref{thetruemanifold} for both $F$ and $R$.
\end{remark}
\subsubsection{From maps to flows}\label{sec:true3}
Putting $t=1$ in Lemma~\ref{lemma:e} we have
$$
F(\widecheck{K}^{\leq}(x,\th,\tau)) - \widecheck{K}^{\leq}(R(x,\th,\tau)) =
\mathcal{O}(|x|^{Q+N},|x|^{Q+N}, |x|^{Q+N-1}).
$$
Then by Theorem~\ref{thetruemanifold}, there exists $\Delta \in \mathcal{X}_{Q+1} \times
\mathcal{X}_{Q+1} \times \mathcal{X}_Q$ such that
$$
F(\widecheck{K}^{\leq} + \Delta) - (\widecheck{K}^{\leq} + \Delta) \circ R = 0,
\qquad \text{in   } S(\beta,\rho) \times \TT_{\sigma'}^{d+d'}
$$
for some parameters $\beta,\rho,\sigma'$. Notice that we have applied Theorem~\ref{thetruemanifold}
with the angles $(\th,\tau)$. Let
$$
\widecheck{K}= \widecheck{K}^\leq + \Delta \qquad \text{and} \qquad  \K^s(x,\th,\tau) = \varphi(-s;\widecheck{K}(\psi(s;x,\th,\tau))).
$$
\begin{lemma}
Given $x,\th,\tau$ belonging to
$S(\beta,\rho) \times \TT_{\sigma'}^{d+d'}$:
\begin{enumerate}
\item $\Delta_{\tau}(x,\th,\tau)=0$, $\widecheck{K}_\tau (x,\th,\tau)=\tau$.
\item $\K^s-\widecheck{K}^{\leq} = \mathcal{O}(|x|^{Q+1},|x|^{Q+1},|x|^Q,0)$.
\item $F\circ \K^s=\K^s\circ R$ and as a consequence, by the uniqueness statement of
Theorem~\ref{thetruemanifold}, $\K^s=\widecheck{K}$ for all $s$.
\end{enumerate}
\end{lemma}
\begin{proof} We start with the first item. Since $F_{\tau}(x,y,\th,\tau) = \varphi_{\tau}(1;x,y,\th,\tau)$,
integrating equation~\eqref{skewvectorfield} we obtain $F_\tau(x,y,\th,\tau)=\tau+\nu
$.
In the same way $R_\tau(x,\th,\tau)=\tau+\nu$. Also
\begin{align*}
0&=F_\tau \circ (\widecheck{K}^\leq + \Delta) - (\widecheck{K}^{\leq}_\tau + \Delta_\tau)\circ R
=\widecheck{K}^{\leq}_\tau + \Delta_\tau + \nu - R_\tau-\Delta_\tau \circ R
\\ &= \Delta_\tau - \Delta_\tau \circ R.
\end{align*}
From this we have $\Delta_\tau = \Delta_\tau \circ R = \Delta_\tau \circ R^j$ for all $j\geq 0$.
Since $\Delta_\tau = \mathcal{O}(|x|^{Q-1})$ and $(R^j)_x$ goes to zero as $j\to \infty$ (see Lemma~\ref{lemaR}) we obtain
$\Delta_\tau\equiv 0$.

To prove the second item, we decompose
$$
\K^s(x,\th,\tau) = \varphi(-s; \widecheck{K}(\psi(s; x,\th,\tau))) = e_1 + e_2,
$$
where
$$
e_1 = \varphi(-s; \widecheck{K}^\leq (\psi(s; x,\th,\tau)))
$$
and
$$
e_2 = \int_{0}^1 D \varphi \big (-s ; \widecheck{K}^\leq (\psi(s; x,\th,\tau)) +
\xi \Delta (\psi(s;x,\th,\tau)) \big ) \Delta (\psi(s;x,\th,\tau))\, d\xi.
$$
By Lemma~\ref{lemma:e} we have
\begin{align*}
e_1 &= \widecheck{K}^{\leq} (\psi(-s; \psi(s;x,\th,\tau))) + e(-s,\psi(s;x,\th,\tau)) \\
 &= \widecheck{K}^{\leq} (x,\th,\tau) + \mathcal{O}(|x|^{Q+N}, |x|^{Q+N} , |x|^{Q+N-1}, 0).
\end{align*}
Since $\partial_\th \varphi_x , \partial_\th \varphi_y, \partial_\tau \varphi_x ,
\partial_\tau \varphi_y$ are $\mathcal{O}(|x|^{N})$,
$\partial_x \varphi_\tau , \partial_y \varphi_\tau , \partial_\th \psi_\tau\equiv 0$
and $\Delta \in \mathcal{X}_{Q+1} \times \mathcal{X}_{Q+1} \times \mathcal{X}_{Q}
\times \{0\}$, we have that
$$
e_2 =\mathcal{O}(|x|^{Q+1}, |x|^{Q+1} , |x|^Q,0).
$$

To prove the third item, we compute
\begin{align*}
F(\K^s(x,\th,\tau))&= \varphi(-s+1; \widecheck{K}(\psi(s;x,\th,\tau)) )
=\varphi(-s; F(\widecheck{K}(\psi(s; x,\th,\tau)))\\
&= \varphi(-s; \widecheck{K}(R(\psi(s; x,\th,\tau)))
=\varphi(-s; \widecheck{K}(\psi(s+1; x,\th,\tau))) \\
&=\varphi(-s; \widecheck{K}(\psi(s; R(x,\th,\tau)))
=\K^s(R(x,\th,\tau))
\end{align*}
and the result is proven.
\end{proof}

Finally, we define $K(x,\th,t)=\widecheck{K}_{x,\th}(x,\th,\nu t)$ and we prove below that it satisfies
the semiconjugation condition for flows, thus providing the parameterization claimed in Theorem~\ref{main_theorem_flow_case}.
\begin{lemma}
We have
\begin{enumerate}
\item $\varphi(s; \widecheck{K}(x,\th, \nu t))=\widecheck{K}(\psi(s; x,\th,\nu t))$.
\item $X(K(x,\th,t),\nu t)= DK(x,\th,t ) Y(x,\th ,t) + \partial_t K(x,\th, t)$.
\end{enumerate}
\end{lemma}
\begin{proof}
(1) follows immediately from the definition of $\K^s$ and the equality
$\K^s = \widecheck{K}$.

For (2) we take derivatives with respect to $s$ on both sides of the equality
in (1) and obtain
\begin{align*}
\widehat{X}(\varphi(s; \widecheck{K}(x,\th,\nu t))) =&
D_{x,\th} \widecheck{K} (\psi_{x,\th}(s; x, \th , \nu t), \psi_{\tau}(s; x,\th , \nu t)) \\
&\times \widecheck{Y}_{x,\th} (\psi_{x,\th}(s; x,\th ,\nu t),\psi_{\tau}(s; x,\th, \nu t)) \\
&+ \partial_\tau \widecheck{K}(\psi(s; x,\th, \nu t)) \cdot \nu,
\end{align*}
where we have used that $\psi_\tau (s;x,\th,\nu t)=\nu (s+t)$.

Taking $s=0$, keeping the components with respect to $x,y$ and $\th$ and
taking into account the definitions of $\widehat{X}, \widehat{Y}, \widehat{K}$ and that $\widecheck{K}_\tau(x,\th, \tau)=\tau$, we finally obtain
$$
X(K(x,\th,t),t) = DK(x,\th, t) Y(x,\th, t) + \partial_t K(x,\th, t).
$$
\end{proof}

\begin{remark}
In the autonomous case, the map $F$ is independent of $\tau$. Then,
if $K^{\leq}$ does not depend on $t$, the parameterization $K$ is also independent of $t$.
\end{remark}
\subsection{Formal parabolic manifold, vector field case. Proof of Theorem~\ref{formal_part_theorem_flows}}\label{sec:proofsedosformal}

We will not write the dependence of the different objects that appear in this section with respect to $\lambda$, but we assume all depend
analytically on $\lambda$.

We prove by induction over $j$ that there exist $K^{(j)	}$ and $Y^{(j)}$. Assuming
the form \eqref{formkvfx}, \eqref{formkvfy}, \eqref{formkvfth}, \eqref{formrvfx} and \eqref{formrvfth} for $K^{(j)}_x$,
$K^{(j)}_y$, $K^{(j)}_\th$, $Y^{(j)}_x$ and $Y^{(j)}_\th$ respectively, we will prove that at the step $j$ we are able to determine the quantities
	$\overline K^j_{x,y}$, $\overline K^{j-1}_{\th}$, $\widetilde K^{	j+N-1}_{x,y}$,
	$\widetilde K^{j+P-2}_{\th}$, $ Y^{j+N-1}_{x}$
	and $ Y^{j+P-2}_{\th}$ so that the order condition
	\eqref{ordreEvf} for the remainder $E^{(j)}  $ is fullfilled.

As for maps, the only case we need to take into consideration is $P\leq N$, since $P>N$ can be deduced from this case by taking $h_N=0$.

We first deal with $j=1$. 	We write
	\begin{align*}
	K^{(1)}_{x}(x,\theta,t)   &= x+
	\widetilde K^{N}_{x}(\theta,t) x^{N}, \quad
	K^{(1)}_{y}(x,\theta,t)   = 0, \quad
	K^{(1)}_{\theta}(x,\theta,t)   = \th,   \\
	&Y^{(1)}_x(x,\theta,t ) =  Y^N_x (\th,t) x^N,  \qquad
	Y^{(1)}_\theta (x,\theta,t) = \omega,
	\end{align*}	
	and we compute $E^{(1)} = X\circ K^{(1)}- D K^{(1)} Y^{(1)}-\partial_t K^{(1)} $. Recall here that $D=\partial_{x,\th}$.
	From the form \eqref{vectorfield_truemanifold} we obtain
	\begin{align*}
	E^{(1)}_{x}(x,\theta,t)   = & [-a(\th,t) -Y^N_x(\th,t) -  \partial_\th \widetilde K^N_x(\th,t) \,\omega
	-\partial _t \widetilde K^N_x(\th,t)]x^N
	+ \O(|x|^{N+1}), \\
	E^{(1)}_{y}(x,\theta,t)   =& \O(|x|^{N+1}) , \qquad E^{(1)}_{\th}(x,\theta,t)   =  \O(|x|^{P}).
	\end{align*}	
	To have $E^{(1)}_{x} (x,\theta,t)= \O(|x|^{N+1})$ we take
	$$
	\overline Y^N_x=-\overline a, \qquad \widetilde Y^N_x = 0, \qquad
	\widetilde K^N_x= -\SDVF (\widetilde  a ).
	$$

	For $j\ge 2$, assuming the induction hypothesis, we write $K^{(j)} = K^{(j-1)} + \mathcal{K}^{(j)}$ and $Y^{(j)} = Y^{(j-1)} + \mathcal{Y}^{(j)}$
	with
	$$
	\mathcal{K}^{(j)} =
	\begin{pmatrix}
	\overline K^{j}_{x} x^{j}+ \widetilde K^{j+N-1}_{x}(\theta,t) x^{j+N-1}  \\
	\overline K^{j}_{y} x^{j}+ \widetilde K^{j+N-1}_{y}(\theta,t) x^{j+N-1}  \\
	\overline K^{j-1}_{\th}  x^{j-1}+	\widetilde K^{j+P-2}_{\th}(\theta,t) x^{j+P-2}
	\end{pmatrix}
	,\qquad
	\mathcal{Y}^{(j)}=
	\begin{pmatrix}
	Y^{j+N-1}_x (\th,t) x^{j+N-1}  \\
	Y^{j+P-2}_\theta  (\th,t) x^{j+P-2}
	\end{pmatrix}.
	$$
	
	Using the induction hypothesis
	\begin{align*}
	{E}^{(j-1)} = &(E_x^{j+N-1}(\th,t) x^{j+N-1} , E_y^{j+N-1}(\th,t) x^{j+N-1} ), E_\th^{j+P-2}(\th,t) x^{x+P-2}  )\\
	&+(\O(|x|^{j+N}), \O(|x|^{j+N}), \O(|x|^{j+P-1}))
	\end{align*}
	and proceeding as in Section~\ref{sec:proofsmapsformal} we conclude that
	\begin{align*}
	E^{(j)}_{x}(x,\theta,t)   =
	& \big [-\partial_\th\widetilde K^{j+N-1}_x (\th,t) \omega-\partial_t\widetilde K^{j+N-1}_x (\th,t)
	+\big(j\overline a -N  a(\th,t)\big)\overline K_x^j  \\ &-
	Y^{j+N-1}_x(\th,t)
-\partial_\th a(\th,t)  \overline K^{j-1}_\th + f_{N-1,1}(\th,t)  \overline K^{j}_{y}
	\\ &+E_x^{j+N-1}(\th,t)
	\big] x^{j+N-1} + \O(|x|^{j+N})\\
	& - \big [ \partial_\th \widetilde K^N_x(\th,t) Y^{j+P-2}_\th(\th,t)
	+\partial_\th a(\th,t)  \widetilde K^{j+P-2}_\th(\th,t) \big ] x^{j+N+P-2}, \\
	E^{(j)}_{y}(x,\theta,t)   =& \big[-\partial_\th\widetilde K^{j+N-1}_y (\th,t) \omega
	-\partial_t\widetilde K^{j+N-1}_y (\th,t) +\big(B(\th,t) + j\overline  a\Id \big) \overline K^{j}_{y}
	\\ &+  E_y^{j+N-1}(\th,t)\big]x^{j+N-1}+ \O(|x|^{j+N}),
	\\
	E^{(j)}_{\th}(x,\theta,t)   = & \big[-\partial_\th\widetilde K^{j+P-2}_\th (\th,t) \omega
	-\partial_t\widetilde K^{j+P-2}_\th (\th,t)-Y^{j+P-2}_\th(\th,t)  \\
	&+  E_\th^{j+P-2}(\th,t)\big] x^{j+P-2} + 	(j-1) \overline a \overline K_\th^{j-1} x^{j+N-2}+\O(|x|^{j+P-1}).
	\end{align*}
	We notice that the above formulae correspond to the ones in~\eqref{Ejx},
	\eqref{Ejy} and~\eqref{Ejth} for maps substituting $R$ by the vector field $Y$ and the operator
	$\widetilde{K}(\th+\omega)-\widetilde{K}(\th)$ by the corresponding infinitesimal version for flows
	$$
	\partial_\th \widetilde{K} \cdot \omega + \partial_t\widetilde{K}
	$$
	mentioned in Section~\ref{sec:sdeqvf}. We recall that the notation $\SD$ has different
	meanings whether it is used in the map or the flow settings, see Sections~\ref{sec:notproofs} and~\ref{sec:sdeqvf} where
	the main features of the small divisors equation in these contexts are exposed.
	As a consequence, the same formulae given along
	Section~\ref{sec:proofsmapsformal} apply in this case. We have indeed:
	$$
	\overline K^{j}_{y}= -[\overline B + j \overline a\Id]^{-1} \overline E_y^{j+N-1} , \qquad
	\widetilde K_y^{j+N-1} = \SD \big(\widetilde B  \cdot \overline K^{j}_{y}+ \widetilde  E_y^{j+N-1}\big).
	$$
When $P<N$
	$$
	Y^{j+P-2}_\th =\overline Y^{j+P-2}_\th=\overline{E}^{j+P-2}_\th ,  \qquad \overline K^{j-1}_\th \; \text{free},
	\qquad
	\widetilde K^{j+P-2}_\th =\SD \big(\widetilde E_\th^{j+P-2}\big)
	$$
	and if $P=N$
	$$
	Y^{j+P-2}_\th =0,\qquad \overline K^{j-1}_\th = -\frac{\overline{E}_\th^{j+P-2} }{(j-1) \overline a},
	\qquad \widetilde K^{j+P-2}_\th=\SD \big(\widetilde E_\th^{j+P-2} \big).
	$$
Defining
	\begin{align*}
	\varphi^{(j)}(\th)=&
	-\partial_\th a(\th) \overline{K}_\th^{j-1} + f_{N-1,1}(\th) \overline{K}_y^j +E_x^{j+N-1} (\th) \\
	\psi^{(j)}(\th)= & \begin{cases}
	\varphi^{(j)}(\th), & P\neq 1, \\
	\varphi^{(j)}(\th) - \partial_\th \widetilde{K}^N(\th+\omega) Y_\th^{j+P-2} - \partial_\th a(\th) \widetilde{K}^{j+P-2}_\th(\th), & P=1,
	\end{cases}
	\end{align*}
if $j = N$ we take
	$$
	Y^{j+N-1}_x = \overline Y^{j+N-1}_x =  \overline{\psi}^{(j)} ,
	\qquad
	\overline K^j_x \quad \text{free}, \qquad
\widetilde K_x^{j+N-1}=  \SD \big(\widetilde{\psi}^{(j)} -N\widetilde{a} \overline K^j_x  \big)
	$$
	and when $j\neq N$,
	$$
	Y^{j+N-1}_x = 0,
	\qquad  \overline K_x^{j}= \frac{\overline{\psi}^{(j)}}{(j-N) \overline a},
	\qquad \widetilde K_x^{j+N-1} = \SD \big (\widetilde{\psi}^{(j)} - N\widetilde{a} \overline{K}_x^j\big ).
	$$
	Moreover all terms depend analytically on $\lambda$.
